\documentclass{amsart}
\usepackage{mathrsfs,amssymb,multirow,setspace,color}
\usepackage{hyperref}
\usepackage[a4paper, total={7in, 9in}]{geometry}
\theoremstyle{plain}
\usepackage{graphicx}
\newtheorem{theorem}{Theorem}[section]
\newtheorem{lemma}[theorem]{Lemma}
\newtheorem{proposition}[theorem]{Proposition}
\newtheorem{corollary}[theorem]{Corollary}

\theoremstyle{definition}

\theoremstyle{remark}
\newtheorem{remark}[theorem]{Remark}

\input xy
\xyoption{all}
\def\a{\mathcal{P}_E(G)}

\def\d{\mathcal{D}(G)}

\begin{document}
\title{Nilpotent groups whose Difference graphs have positive genus}


\author[Parveen, Jitender Kumar]{Parveen, Jitender Kumar$^{*}$}
\address{$\text{}^1$Department of Mathematics, Birla Institute of Technology and Science Pilani, Pilani 333031, India}
\email{p.parveenkumar144@gmail.com,jitenderarora09@gmail.com}

\begin{abstract}
The power graph of a finite group $G$ is a simple undirected graph with vertex set $G$ and two vertices are adjacent if one is a power of the other. The enhanced power graph of a finite group $G$ is a simple undirected graph whose vertex set is the group $G$ and two vertices $a$ and $b$ are adjacent if there exists $c \in G$ such that both $a$ and $b$ are powers of $c$.  In this paper, we study the difference graph $\mathcal{D}(G)$ of a finite group $G$ which is the difference of the enhanced power graph and the power graph of $G$ with all isolated vertices removed. We characterize all the finite nilpotent groups $G$ such that the genus (or cross-cap) of the difference graph $\mathcal{D}(G)$ is at most $2$.
 \end{abstract}

\subjclass[2020]{05C25}

\keywords{Enhanced power graph, power graph, nilpotent groups, genus and cross-cap of a graph. \\ *  Corresponding author}

\maketitle
\section{Historical Background and Main result}
There are number of graphs attached to groups, in particular: Cayley graphs, commuting graphs, power graphs, enhanced  power graph etc. These graphs have been studied extensively by various researchers because of their vast applications (see \cite{a.hayat2019novel,kelarev2003graph,kelarev2004labelled,a.kelarev2009cayley}). The \emph{power graph} $\mathcal{P}(G)$ of a finite group $G$ is a simple undirected graph with vertex set $G$ such that two vertices $a$ and $b$ are adjacent if one is a power of the other. Kelarev and Quinn \cite{a.kelarev2004combinatorial} introduced the concept of directed power graph. Topological graph theory is mainly concerned with embedding of a graph into a surface without edge crossings. Its applications lie in electronic printing circuits where the purpose is to embed a circuit, that is, the graph on a circuit board (the surface), without two connections crossing each other, resulting in a short circuit. The problem of determining the genus of a graph is NP-hard \cite{a.Thomassen1989}. Mirzargar \emph{et al.} \cite{a.Mirzargar2012} classified all the finite groups with planar power graphs. Further, Doostabadi \emph{et al.} \cite{a.doostabadi2017}, characterized the finite groups whose power graphs are of (non)orientable genus one. Then all the finite groups with (non)orientable genus two power graphs have been characterized in \cite{a.ma2019power}. The undirected power graphs of groups have been studied in other aspects, see \cite{a.Cameron2010,a.Cameron2011,a.Feng2015metricdimmension,a.powergraphsurvey,a.feng2016rainbow} and references therein. 
The \emph{commuting graph} $\Delta(G)$ of a group $G$ is a simple graph with vertex set $G$ and two distinct vertices $x, y$ are adjacent if $xy = yx$. Clearly, all the  central elements are the dominating vertices in $\Delta(G)$. Moreover, $\Delta(G)$ has been investigated by various researchers by taking its vertex set as non-central elements. Afkhami \emph{et al.} \cite{a.Afkhami2015} characterized all the finite groups whose commuting and noncommuting graphs are planar, projective planar and of genus one, respectively.
Results on the commuting graph associated to groups can be found in \cite{a.britnell2013perfect,a.Haji2019,a.iranmanesh2008,a.kumar2021,a.Mohmoudifar2017,a.Zhai2022} and references therein.  Aalipour \emph{et al.} \cite{a.Cameron2016} characterize the finite group $G$ such that the power graph $\mathcal{P}(G)$ and the commuting graph $\Delta(G)$ are not equal and hence they introduced a new graph between power graph and commuting graph, called enhanced power graph. The \emph{enhanced power graph} $\a$ of a finite group $G$ is a simple undirected graph with vertex set $G$ and two vertices $x$ and $y$ are adjacent if $x,y\in \langle z \rangle$ for some $z\in G$. Equivalently, two vertices $x$ and $y$ are adjacent in $\a$ if and only if $\langle x,y\rangle$ is a cyclic subgroup of $G$. The clique number of enhanced power graph of an arbitrary group $G$ was obtained by Aalipour et al. \cite{a.Cameron2016} in terms of orders of elements of $G$. Bera and Bhuniya \cite{a.Bera2017} proved that $\a$ is planar if and only if order of each element of $G$ is at most $4$. However, all the finite groups having genus one (or two) enhanced power graphs have not been classified so far. For a detailed list of results and open problems on enhanced power graphs of groups, we refer the reader to \cite{a.masurvey2022}. 



From the above definitions, it is easy to observe that the power graph is a spanning subgraph of the enhanced power graph. Also, the enhanced power graph is a spanning subgraph of the commuting graph. Consequently, Aalipour \emph{et al.} {\rm \cite[Question 42]{a.Cameron2016}} motivated the researchers to study the connectedness of the difference graph of the commuting graph and power graph of a group $G$ i.e. the graph with vertex set $G$ in which $x$ and $y$ are adjacent if they commute but neither is a power of the other. Further, Cameron \cite{a.camerongraphdefinedongroups2022} discussed some developments on the difference graph $\Delta(G) - \mathcal{P}(G)$. Moreover, some results on the difference graph of the commuting graph and the enhanced power graph of a group have been discussed in \cite{a.camerongraphdefinedongroups2022}. Motivated by the work of \cite{a.camerongraphdefinedongroups2022}, Biswas \emph{et al.} \cite{a.biswas2022difference} studied the difference graph $\d = \mathcal{P}_{E}(G) - \mathcal{P}(G)$ of enhanced power graph and power graph of a finite group $G$ with all isolated vertices removed. For certain group classes, the connectedness and perfectness of $\d$  has been investigated in \cite{a.biswas2022difference}. Together with the planarity, various forbidden graph classes of $\d$ have been studied in \cite{a.kumar2022difference}. The purpose of this article is to classify all the finite nilpotent groups such that the difference graph $\d$ is of genus (or cross-cap) at most two. If $G$ is a $p$-group, then it is well known that the power graph and enhanced power graph of $G$ are equal. Thus, $\d$ is a null graph, whenever $G$ is a $p$-group. Before providing our main result of this paper, for a finite group $G$, first we define\\
\begin{itemize}
    \item \emph{$G$ satisfies the condition $\mathcal{C}_1$}, if $G\cong P\times \mathbb{Z}_3$, where  $P$ is a $2$-group with exponent $4$. Moreover, $P$ contains two maximal cyclic subgroups $H$ and $K$ of order $4$ such that $|H\cap K|=2$, and the intersection of any other pair of maximal cyclic subgroups of $P$ is trivial.
    
     \item \emph{$G$ satisfies the condition $\mathcal{C}_2$}, if $G\cong P\times \mathbb{Z}_3$, where  $P$ is a $2$-group with exponent $4$. Moreover, $P$ contains four maximal cyclic subgroups $H_1, H_2, H_3$ and $H_4$ of order $4$ such that $|H_1\cap H_2|=|H_3\cap H_4|=2$, and the intersection of any other pair of maximal cyclic subgroups of $P$ is trivial.

      \item \emph{$G$ satisfies the condition $\mathcal{C}_3$}, if $G\cong P\times \mathbb{Z}_3$, where  $P$ is a $2$-group with exponent $4$.  Moreover, $P$ contains three maximal cyclic subgroups $H_1, H_2$ and $H_3$ of order $4$ such that $|H_1\cap H_2\cap H_3|=2$, and the intersection of any other pair of maximal cyclic subgroups of $P$ is trivial.
\end{itemize}
For $1\leq i \leq 3$, if the group $G$ satisfies the condition $\mathcal{C}_i$, then we write it by the group $\mathcal{G}_i$.
The main result of this paper is as follows.
\begin{theorem}{\label{Main}}
    Let $G$ be a nilpotent group which is not a $p$-group and let $\gamma(\d)$ and $\overline{\gamma}(\d)$ be the genus and cross-cap of $\d$, respectively. Then
    \begin{itemize}
        \item[(i)] $\gamma(\d)=1$ if and only if $G$ is isomorphic to one of the following groups: $$\mathbb{Z}_{18}, \ \mathbb{Z}_{20}, \  \mathbb{Z}_2\times \mathbb{Z}_2 \times  \mathbb{Z}_5, \ \mathbb{Z}_{28}, \  \mathbb{Z}_2\times \mathbb{Z}_2 \times  \mathbb{Z}_7, \   \  \mathcal{G}_1.$$
         \item[(ii)] $\gamma(\d)=2$ if and only if $G$ is isomorphic to one of the following groups: $$\mathbb{Z}_{35}, \ \mathbb{Z}_4\times \mathbb{Z}_3 \times  \mathbb{Z}_3, \ \mathbb{Z}_2\times \mathbb{Z}_2 \times  \mathbb{Z}_3 \times  \mathbb{Z}_3,  \ \mathbb{Z}_2\times \mathbb{Z}_2 \times  \mathbb{Z}_{11}, \  \mathbb{Z}_{44}, \   \mathcal{G}_2, \  \mathcal{G}_3.$$
          \item[(iii)] $\overline{\gamma}(\d)=1$ if and only if $G$ is isomorphic to $\mathbb{Z}_{20}$ or $\mathbb{Z}_2\times \mathbb{Z}_2 \times  \mathbb{Z}_5$.
           \item[(iv)]  $\overline{\gamma}(\d)=2$  if and only if $G$ is isomorphic to one of the groups:  $\mathbb{Z}_{18}$, $\mathbb{Z}_{28}$,  $\mathbb{Z}_2\times \mathbb{Z}_2 \times  \mathbb{Z}_7$,    $\mathcal{G}_1$.
    \end{itemize}
\end{theorem}

\section{Preliminaries}
In this section, we recall the necessary definitions, results  which we need in the sequel of this paper. We also fix our notations in this section. 
Let $G$ be a group. The order of an element $x\in G$ is denoted by $o(x)$ and we denote $\pi _G=\{o(g): g \in G\}$.  By $\langle x, y\rangle$, we mean the subgroup of $G$ generated by $x$ and $y$. The \emph{exponent} $\mathrm{exp}(G)$ of a finite group $G$ is defined as the least common multiple of the orders of all the elements of $G$. For $d\in \pi _G$, $C_d$ denotes the number of cyclic subgroups of order $d$ in $G$.
A cyclic subgroup of a group $G$ is called a \emph{maximal cyclic subgroup} if it is not properly contained in any cyclic subgroup of $G$ other than itself. Note that if $G$ is a cyclic group, then $G$ is  the only maximal cyclic subgroup of $G$. A finite group $G$ is called a \emph{p-group} if $|G|=p^{\alpha}$ for some prime $p$. A finite group $G$ is said to be an EPPO-group if the order of each element of $G$ is of prime power. Otherwise, $G$ is called a non-EPPO-group. The following results are useful for later use.
\begin{lemma}{\label{lemma 3}} Let $G$ be a finite $p$-group with exponent $p^2$. Then either $G$ has exactly one cyclic subgroup of order $p^2$ or $G$ contains at least two cyclic subgroups $M$ and $N$ of order $p^2$ such that $|M\cap N|=p$.
\end{lemma}
\begin{proof}
    First note that $|Z(G)|\geq p$. If $G$ has exactly one cyclic subgroup of order $p^2$, then there is nothing to prove. We may now suppose that $G$ has two cyclic subgroups $M$ and $N$ of order $p^2$. Let $x\in Z(G)$ such that $o(x)=p$. If $x\in M\cap N$, then $|M\cap N|=p$. Now assume that $x\notin M$ and $M=\langle y \rangle$. We claim that $o(xy)=p^2$.  Clearly, $(xy)^{p^2}=x^{p^2}y^{p^2}=e$. Consequently,  $o(xy)\vert p^2$. Thus, $o(xy)\in \{1,p,p^2\}$. If $o(xy)=1$, then $x=y^{-1}$; a contradiction. If $o(xy)=p$, then $x^py^p=ey^p=e$, again a contradiction. It follows that $o(xy)=p^2$. If $\langle xy\rangle=\langle y \rangle$, then $xy=y^k$ for some positive integer $k$ and so $x=y^{k-1}$, which is not possible. Thus, $M'=\langle xy \rangle$ is a cyclic subgroup of order $p^2$ in $G$. Moreover, $M\cap M'=\{e,y^p, y^{2p},\ldots , y^{(p-1)p}\}$. Thus, the result holds.
\end{proof}
\begin{theorem}{\rm \cite{a.pgroupberkovi,b.pgroupisaac2006,a.pgroupkulakoff,a.pgroupmiller}}{\label{pgroupclass}}
Let $G$ be a finite $p$-group of exponent $p^{k}$. Assume that $G$ is not cyclic for an odd prime $p$, and for $p=2$ it is neither cyclic nor of maximal class. Then
\begin{enumerate}
    \item[(i)] $C_p \equiv 1+p\left(\bmod \ p^{2}\right)$.
    \item[(ii)] $p \mid C_{p^i}$ for every $2 \leq i \leq k$.
\end{enumerate}
\end{theorem}
\begin{corollary}{\rm \cite{a.sarkar2022lambda}}{\label{pgroupclasscorollary}}
Let $G$ be a finite $p$-group of exponent $p^{k}$. Then $C_{p^{i}}=1$, for some $1 \leq i \leq k$, if and only if one of the following occurs:
\begin{enumerate}
    \item $G \cong \mathbb{Z}_{p^{k}}$ and $C_{p^{j}}=1$ for all $1 \leq j \leq k$, or
    \item $p=2$ and $G$ is isomorphic to one of the following $2$-groups:
\end{enumerate}
\begin{enumerate}
    \item[(i)] dihedral $2$-group
$$
\mathbb{D}_{2^{k+1}}=\left\langle x, y: x^{2^{k}}=1, y^{2}=1, y^{-1} x y=x^{-1}\right\rangle, \quad(k \geq 1)
$$
where $C_2=1+2^{k}$ and $C_{2^{j}}=1 \text{ for all } (2 \leq j \leq k)$.
\item[(ii)] generalized quaternion $2$-group
$$
\mathbb{Q}_{2^{k+1}}=\left\langle x, y: x^{2^{k}}=1, x^{2^{k-1}}=y^{2}, y^{-1} x y=x^{-1}\right\rangle, \quad(k \geq 2)
$$
where $C_4=1+2^{k-1}$ and $C_{2^{j}}=1$ for all $1 \leq j \leq k$ and $j \neq 2$.
\item[(iii)] semi-dihedral $2$-group
$$
\mathbb{SD}_{2^{k+1}}=\left\langle x, y: x^{2^{k}}=1, y^{2}=1, y^{-1} x y=x^{-1+2^{k-1}}\right\rangle, \quad(k \geq 3)
$$
where $C_2=1+2^{k-1}, C_4=1+2^{k-2}$ and $C_{2^{j}}=1$ for all $3 \leq j \leq k$.
\end{enumerate}
\end{corollary}
In view of Theorem \ref{pgroupclass} and Corollary \ref{pgroupclasscorollary}, we have the following lemma:
\begin{lemma}{\label{lemma 4}}
    Let $G$ be a finite $p$-group with exponent $p^2$ and $G$ contains exactly one cyclic subgroup of order $p^2$. Then the following holds:
    \begin{itemize}
        \item[(i)] If $p=2$, then $G$ is isomorphic to $\mathbb{Z}_4$ or $\mathbb{D}_8$.
        \item[(ii)] If $p>2$, then $G$ is isomorphic to $\mathbb{Z}_{p^2}$.
    \end{itemize}
\end{lemma}

\begin{theorem}{\rm \cite{b.dummit1991abstract}}{\label{nilpotent}}
 Let $G$ be a finite group. Then the following statements are equivalent:
 \begin{enumerate}
     \item[(i)] $G$ is a nilpotent group.
     \item[(ii)] Every Sylow subgroup of $G$ is normal.
    \item[(iii)] $G$ is the direct product of its Sylow subgroups.
    \item[(iv)] For $x,y\in G, \  x$ and $y$ commute whenever $o(x)$ and $o(y)$ are relatively primes.
 \end{enumerate}
 \end{theorem}
 \begin{lemma}{\rm \cite[Lemma 2.5]{a.dalal2022lambda}}{\label{nilpotent lcm}}
     Let $G$ be a finite nilpotent group and $x,y \in G$ be such that $o(x)=s$ and  $o(y)=t$. Then there exists an element $z\in G$ such that $o(z)=\mathrm{lcm}(s,t)$.
 \end{lemma}

Now we recall the necessary graph theoretic definitions and notions from  \cite{b.godil,b.westgraph}. A \emph{graph} $\Gamma$ consists of a vertex set $V(\Gamma)$ and an edge set $E(\Gamma)$, where an edge is an unordered pair of distinct vertices of $\Gamma$. If $\{u_1,u_2\}$ is an edge then, we say that $u_1$ is \emph{adjacent} to $u_2$ and write it as $u_1 \sim u_2$. Otherwise, we denote by $u_1 \nsim u_2$. An edge $\{u,v\}$ in a graph $\Gamma$ is called a \emph{loop} if $u=v$. A graph with no loops or multiple edges is called a \emph{simple} graph. In this paper, we are considering only simple graphs. Let $\Gamma$ be a graph. If both the sets $V(\Gamma)$ and $E(\Gamma)$ are empty, then $\Gamma$ is called a \emph{null graph}.  A \emph{subgraph} of  $\Gamma$ is a graph $\Gamma'$ such that $V(\Gamma') \subseteq V(\Gamma)$ and $E(\Gamma') \subseteq E(\Gamma)$. A subgraph $\Gamma '$ of $\Gamma$ is an \emph{induced subgraph} if two vertices of $V(\Gamma')$ are adjacent in $\Gamma'$ if and only if they are adjacent in $\Gamma$. A graph is called \emph{complete} if every pair of distinct vertices are adjacent. The complete graph on $n$ vertices is denoted by $K_n$. A graph $\Gamma$ is said to be $k$-partite if the vertex set of $\Gamma$ can be partitioned into $k$ subsets such that no two vertices in the same subset being adjacent. If $k=2$, then $\Gamma$ is called a \emph{bipartite} graph. A \emph{complete k-partite} graph, denoted by $K_{n_1,n_2,\ldots ,n_k}$, is a $k$-partite graph having its parts sizes $n_1,n_2,\ldots ,n_k$ such that every vertex in each part is adjacent to all the vertices of all other parts of $K_{n_1,n_2,\ldots ,n_k}$.  A \emph{walk} $\lambda$ in a graph $\Gamma$ from the vertex $u$ to the vertex $v$ is a sequence of vertices $u=u_1, u_2,\ldots , u_m = v(m >1)$ such that $u_i\sim u_{i+1}$ for every $i\in \{1,2,\ldots ,m-1\}$. If no edge is repeated in $\lambda$, then it is called a \emph{trail} in $\Gamma$. If no vertex is repeated in $\lambda$, then it is called a \emph{path} in $\Gamma$. 
A graph $\Gamma$ is \emph{connected}  if every pair of vertices has a path in $\Gamma$. Otherwise, $\Gamma$ is \emph{disconnected}. 
   A graph $\Gamma$ is \emph{planar} if it can be drawn on a plane without edge crossing. A planar graph is said to be \emph{outerplanar} if it can be drawn in the plane such that all its vertices lie on the outer face.  A compact connected topological space  such that each point has a neighbourhood homeomorphic to an open disc is called a \emph{surface}. A graph is said to be \emph{embeddable} on a topological surface if it can be drawn on the surface without edge crossing. Let $\mathbb{S}_{g}$ be an orientable surface with $g$ handles, where $g$ is a non-negative integer. The genus $\gamma(\Gamma)$ of a graph $\Gamma$,  is the minimum integer $g$ such that the graph can be embedded in $\mathbb{S}_{g}$, i.e. the graph $\Gamma$ can be drawn into the surface $\mathbb{S}_{g}$ with no edge crossing. Note that the graphs having genus $0$ are planar, and the graphs having genus one are toroidal.
   Let $\mathbb{N}_k$ be the non-orientable surface formed by connected sum of $k$ projective planes, that is, $\mathbb{N}_k$ is a non-orientable surface with $k$ cross-cap. The \emph{cross-cap} $\overline{\gamma}(\Gamma)$ of a graph $\Gamma$, is the minimum non-negative integer $k$ such that $\Gamma$ can be embedded in $\mathbb{N}_k$. For instance, a graph $\Gamma$ is planar if $\overline{\gamma}(\Gamma)=0$ and, $\Gamma$ is \emph{projective-planar} if $\overline{\gamma}(\Gamma)=1$.
The following results are useful for later use.
    \begin{theorem}{\cite{b.westgraph}}{\label{bipartitecondition}}
A  graph $\Gamma$ is bipartite if and only if it has no odd cycle.
\end{theorem}
\begin{theorem}{\cite{b.white1985graphs}}{\label{planarcondition}} The genus and cross-cap of the complete graphs $K_n$ and $K_{m, n}$ are given below:
\begin{itemize}
\item[(i)] $\gamma(K_n)= \left\lceil{\frac{(n-3)(n-4)}{12}}\right\rceil $, $n\geq 3$.
\item[(ii)]$\gamma(K_{m,n})=\left\lceil \frac{(m-2)(n-2)}{4}\right\rceil $, $m,n\geq 2$.
\item[(iii)] $\overline{\gamma}(K_n)= \left\lceil{\frac{(n-3)(n-4)}{6}}\right\rceil $, $n\geq 3$, $n\neq 7$; $\overline{\gamma} (K_7)=3$.
\item[(iv)] $\overline{\gamma}(K_{m,n})=\left\lceil \frac{(m-2)(n-2)}{2}\right\rceil $, $m,n\geq 2$.
\end{itemize}
\end{theorem}

\begin{remark}{\label{order not divide}}
Let $x$ and $y$ be two elements of a finite group $G$ such that neither $o(x)\vert o(y)$ nor $o(y)\vert o(x)$. Then $x\nsim y$ in $\mathcal{P}(G)$. The converse is also true if $x$ and $y$ belong to the same cyclic subgroup of $G$.
\end{remark}    
\begin{proposition}{\rm \cite[Proposition 2.1]{a.biswas2022difference}}{\label{proposition 2.1}}
    Let $G$ be a non-trivial finite group and  $g\in G$ be a non-identity element. Then $g\notin V(\d)$ if and only if either $\langle g\rangle$ is a maximal cyclic subgroup of $G$, or every cyclic subgroup of $G$ containing $g$ has prime-power order.
\end{proposition}
\begin{proposition}{\rm \cite[Proposition 4.1]{a.biswas2022difference}}{\label{nilpotent coprime adj}}
Let $G$ be a finite nilpotent group and $x,y$ be two non-identity elements of $G$ such that $\mathrm{gcd}(o(x),o(y))=1$, then $x\sim y$ in $\d$.
\end{proposition}
\begin{proposition}{\rm \cite[Corollary 3.6]{a.kumar2022difference}}{\label{x nonadj y in P}}
Let $G$ be a finite nilpotent group and $x$, $y$ be two non-identity elements of a Sylow subgroup of $G$. Then $x\nsim y$ in $\d$.

\end{proposition}
\begin{lemma}{\rm \cite[Lemma 4.4]{a.biswas2022difference}}{\label{induced lemma}}
Let $G$ be a group and let $H$ be a non-EPPO-subgroup of $G$. Then $\mathcal{D}(H)$ is an induced subgraph of $\d$.
\end{lemma}

    

\begin{theorem}{\rm \cite[Theorem 4.7]{a.kumar2022difference}}{\label{PlanarD}}
Let $G$ be a finite nilpotent group which is not a $p$-group. Then $\d$ is planar if and only if $G$ is isomorphic to one of the following groups: 
\[\mathbb{Z}_{12}, \;\mathbb{D}_8\times \mathbb{Z}_3, \; \mathbb{Z}_2 \times \mathbb{Z}_2 \times \cdots \times \mathbb{Z}_2 \times \mathbb{Z}_3, \; \mathbb{Z}_2 \times Q_1, \; \mathbb{Z}_3 \times Q_2, \; \text{where}\]
 $Q_i$ is a $q_i$-group of prime exponent $q_i$ such that $q_1> 2$ and $q_2 > 3$.
\end{theorem}


\section{Proof of the main result}
In this section, we provide the proof of our main result.  Let $G=P_1P_2\cdots P_r$, where $P_i$ $(1 \le i \le r)$ is a Sylow $p_i$-subgroup of $G$ and $|P_i|=p_i^{\alpha _i}$, be a nilpotent group. For $x\in G$, there exists a unique element $x_i\in P_i$ for each $i\in [r]=\{1,2,\ldots ,r\}$ such that $x=x_1x_2\cdots x_r$. Since $(x_1,x_2,\ldots , x_r) \longmapsto x_1x_2\cdots x_r$ is a group isomorphism from $P_1\times P_2\times \cdots \times P_r$ to $P_1P_2\cdots P_r$. Thus, instead of $P_1P_2\cdots P_r$, we sometimes write $P_1\times P_2\times \cdots \times P_r$  without referring to it. In order to prove Theorem \ref{Main}, the following lemmas are useful.  

\begin{lemma}{\label{lemma 1}}
    Let $G=P_1P_2\cdots P_r \ (r\geq 2)$ be a finite nilpotent group. Then the subgraph of $\d$ induced by the set $(P_1\cup P_2\cup \cdots \cup P_r)\setminus \{e\}$ is isomorphic to complete $r$-partite graph $K_{|P_1|-1,|P_2|-1,\ldots , |P_r|-1}$.
\end{lemma}

\begin{proof}
The result holds by Propositions \ref{nilpotent coprime adj} and \ref{x nonadj y in P}. 
\end{proof}

\begin{lemma}{\label{lemma 2}}
Let $G\cong P_1\times P_2$ be a finite  nilpotent group, where $\mathrm{exp}(P_i)=p_i$ for each $i\in \{1,2\}$. Then $\d$ is a complete bipartite graph isomorphic to $K_{|P_1|-1, |P_2|-1}$.
\end{lemma}
\begin{proof}
    First note that $\pi _G=\{1,p_1,p_2,p_1p_2\}$ and the identity element of $G$ does not belongs to $V(\d)$. By Proposition \ref{proposition 2.1}, the elements of the order $p_1p_2$ do not belong to $V(\d)$. Thus, $V(\d)$ contains the elements of orders $p_1$ and $p_2$ only. Consider the sets $A=\{x\in V(\d) : o(x)=p_1\}$ and $B=\{y\in V(\d) : o(y)=p_2\}$. Clearly, $A$ and $B$ forms a partition of $V(\d)$. By Propositions \ref{nilpotent coprime adj} and \ref{x nonadj y in P}, $\d$ is a complete bipartite graph which is isomorphic to $K_{|P_1|-1, |P_2|-1}$. This completes our proof.
\end{proof}
\begin{lemma}{\label{lemma 5}}
     Let $G=P_1P_2\cdots P_r$ be a finite nilpotent group. If $r\geq 3$, then $\gamma(\d)\geq 3$ and $\overline{\gamma}(\d)\geq 6$.
\end{lemma}
\begin{proof}
    Let $G=P_1P_2\cdots P_r$ be a nilpotent group such that $|P_i|=p_i^{\alpha _i}$ and $p_j<p_{j+1}$ for $j\in [r-1]$. We prove our result in the following two cases: 
   \vspace{.2cm}
  
\noindent\textbf{Case-1:} $r\geq 4$. By Lemma \ref{lemma 1}, the subgraph of $\d$ induced by the set $(P_1\cup P_2\cup \cdots \cup P_r)\setminus \{e\}$ contains a subgraph isomorphic to $K_{1,2,4,6}$. Note that $K_{1,2,4,6}$ has a subgraph which is isomorphic to $K_{7,6}$. By Theorem \ref{planarcondition}, we get $\gamma(\d)\geq 5$ and $\overline{\gamma}(\d)\geq 10$.  
\vspace{.1cm}

\noindent\textbf{Case-2:} $r=3$. In this case, we have $p_3\geq 5$. Now, we discuss the following subcases: 
     \vspace{.05cm}
     
   \textbf{Subcase-2.1:} $p_3=5$. Clearly, $p_1=2$ and $p_2=3$. By Lemma \ref{nilpotent lcm}, $G$ has an element $x$ such that $\langle x \rangle \cong \mathbb{Z}_{30}$. Note that $\mathbb{Z}_{30}$ has $4$ elements of order $10$ and $8$ elements of order $15$. Suppose $x_i, y_j\in \langle x\rangle$ such that $o(x_i)=10$ and $o(y_j)=15$ for every $i\in [4]$, $j\in [8]$. Note that  $x_i\sim y_j$ in $\a$, but they are not adjacent in $\mathcal{P}(G)$ (see Remark \ref{order not divide}). It follows that for each $i\in [4]$ and  $j\in [8]$, we have $x_i\sim y_j$ in $\d$. Consequently, the subgraph of $\d$ induced by the set $\{x_1,x_2,x_3 , x_4,y_1,\ldots , y_8\}$ is isomorphic to $K_{4,8}$. Thus, $\gamma(\d)\geq 3$ and $\overline{\gamma}(\d)\geq 6$. 
    \vspace{.05cm}
    
   \textbf{Subcase-2.2:} $p_3\geq 7$. In this subcase $G$ has at least $1$ element of order $p_1$, $2$ elements of order $p_2$,  $2$ elements of order $p_1p_2$ and $6$ elements of order $p_3$. By Proposition \ref{nilpotent coprime adj}, the elements of  order $p_1$, $p_2$ and $p_1p_2$ will be adjacent to each element of order $p_3$. Consequently, $\d$ contains a subgraph isomorphic to $K_{5,6}$ and so $\gamma(\d)\geq 3$, $\overline{\gamma}(\d)\geq 6$.
   \end{proof} Now we prove our main result of this manuscript.\\ 
\textbf{Proof of Theorem \ref{Main}.} Let $G=P_1P_2\cdots P_r$ be a nilpotent group such that $|P_i|=p_i^{\alpha _i}$ and $p_j<p_{j+1}$ for $j\in [r-1]$. First, suppose that $\gamma(\d)\leq 2$ and $\overline{\gamma}(\d)\leq 2$.  By Lemma \ref{lemma 5}, we obtain $r=2$. Thus, $G=P_1\times P_2$. We prove our result through the following cases:
 
 \noindent\textbf{Case-1:} $p_1\geq 7$. It follows that $p_2\geq 11$. By Lemma \ref{lemma 1}, the subgraph of $\d$ induced by the set $(P_1  \cup P_2)\setminus \{e\}$ has a subgraph isomorphic to $K_{6,10}$. Consequently, $\gamma(\d)\geq 8$ and $\overline{\gamma}(\d)\geq 16$.
 \vspace{.1cm}
 
    \noindent\textbf{Case-2:} $p_1=5$. Then we must have $p_2\geq 7$.

    \vspace{.05cm}
      \textbf{Subcase-2.1:} \emph{$|P_1|= 5$ and  $|P_2|= p_2$}. Then $G\cong \mathbb{Z}_{5}\times \mathbb{Z}_{p_2}$. By Lemma \ref{lemma 2}, $\d\cong K_{4,p_2-1}$. If $p_2=7$, then $\gamma(\mathcal{D}(\mathbb{Z}_{35}))=2$ and $\overline{\gamma}(\mathcal{D}(\mathbb{Z}_{35}))= 4$. If $p_2\geq 11$, then $\gamma(\d)\geq 4$ and $\overline{\gamma}(\d)\geq 8$.

     \vspace{.05cm}
      \textbf{Subcase-2.2:} \emph{$|P_1|= 5^{\alpha}$ and  $|P_2|= p_2^{\beta}$, where both $\alpha$ and $\beta$ are not equal to $1$}. If $\alpha \geq 2$, then the graph induced by the set $(P_1\cup P_2)\setminus \{e\}$ has a subgraph isomorphic to $K_{24,6}$. Consequently, $\gamma(\d)\geq 22$ and $\overline{\gamma}(\d)\geq 44$. Similarly, if $\alpha =1$ and $\beta \geq 2$, then we get $\gamma(\d)\geq 23$ and $\overline{\gamma}(\d)\geq 46$.
       \vspace{.1cm}
       
  \noindent\textbf{Case-3:} $p_1=3$. Clearly, $p_2\geq 5$.
  \vspace{.05cm}
  
     \textbf{Subcase-3.1:} \emph{$|P_1|=3$}. Note that $\mathrm{exp}(P_2)=p_2^{\alpha}$, where $\alpha \geq 1$. If $\alpha=1$, then by Theorem \ref{PlanarD}, the graph $\d$ is planar. Consequently, $\gamma(\d)=0=\overline{\gamma}(\d)$. Now we assume that $\alpha \geq 2$. Let $x\in P_1$ such that $o(x)=3$ and $y\in P_2$ such that $o(y)=p_2^2$. Notice that $\langle xy \rangle$ is a cyclic subgroup of order $3p_2^2$ in $G$. Consider the sets $S=\{z\in \langle xy \rangle : o(z)=3p_2\}$ and $T=\{z'\in \langle xy \rangle : o(z')=p_2^2\}$. Let $x'\in S$ and $y'\in T$. Clearly, $x'\sim y'$ in $\a$. Also, neither $o(x')\vert o(y')$ nor $o(y')\vert o(x')$. By Remark \ref{order not divide}, $x'\nsim y'$ in $\mathcal{P}(G)$. Thus, $x'\sim y'$ in $\d$. Consequently, $\d$ contains a subgraph isomorphic to $K_{|S|,|T|}$. Since $p_2\geq 5$, we obtain $|S|\geq 8$ and $|T|\geq 20$. Thus, $\gamma(\d)\geq 27$ and $\overline{\gamma}(\d)\geq 54$.

     \vspace{.05cm}
     \textbf{Subcase-3.2:} \emph{$|P_1|= 3^{\alpha}$, where $\alpha \geq 2$}. Notice that $P_1$ has at least $8$ non-identity elements and $P_2$ has at least $4$ non-identity elements. By Lemma \ref{lemma 1}, the graph induced by $(P_1\cup P_2)\setminus \{e\}$ contains a subgraph isomorphic to $K_{8,4}$. Consequently, $\gamma(\d)\geq 3$ and $\overline{\gamma}(\d)\geq 6$.
      \vspace{.1cm}
       
  \noindent\textbf{Case-4:} $p_1= 2$. Now, we have the following possible subcases. 

     \vspace{.05cm}
     \textbf{Subcase-4.1.1:} \emph{$|P_1|=2$ and $\mathrm{exp}(P_2)=p_2$}. By Theorem \ref{PlanarD}, the graph $\d$ is planar. Consequently, $\gamma(\d)=0=\overline{\gamma}(\d)$.

\vspace{.05cm}
     \textbf{Subcase-4.1.2:} \emph{$|P_1|=2$, $|P_2|=3^{\alpha}$, where $\alpha \geq 2$ and $\mathrm{exp}(P_2)=9$}. In view of Lemma  \ref{lemma 3}, we have the following two further subcases:

   \vspace{.05cm}
       \; \; \; \textbf{Subcase-4.1.2(a):} \emph{$P_2$ conatins exactly one cyclic subgroup of order $9$}. By Lemma \ref{lemma 4}, we get $G\cong \mathbb{Z}_{18}$. 
     Observe that the subgraph of $\mathcal{D}(\mathbb{Z}_{18})$ induced by the set $\{3,9,15,2,4,8,10,14,16\}$ is isomorphic to $K_{3,6}$. Thus, $\gamma(\d)\geq 1$ and $\overline{\gamma}(\d)\geq 2$.
     A genus $1$ and cross-cap $2$ drawings of $\mathcal{D}(\mathbb{Z}_{18})$ are given in Figures \ref{Z18} and \ref{Z18c}, respectively.

     \begin{figure}[ht]
    \centering
    \includegraphics[width=0.50\textwidth]{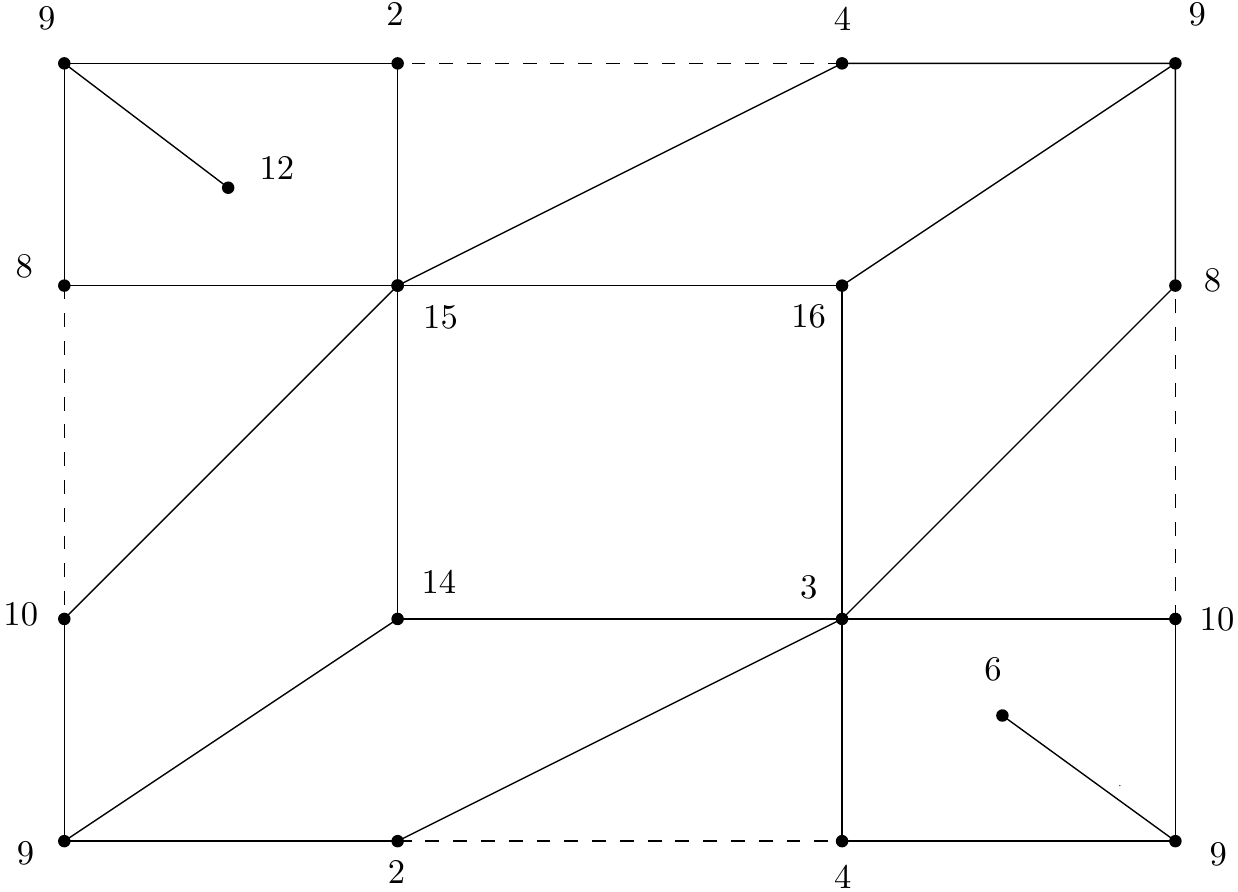}
    \caption{Embedding of ${\mathcal{D}(\mathbb{Z}_{18})}$ in $\mathbb{S}_1$.}
    \label{Z18}
\end{figure}

 \begin{figure}[ht]
    \centering
    \includegraphics[width=0.50\textwidth]{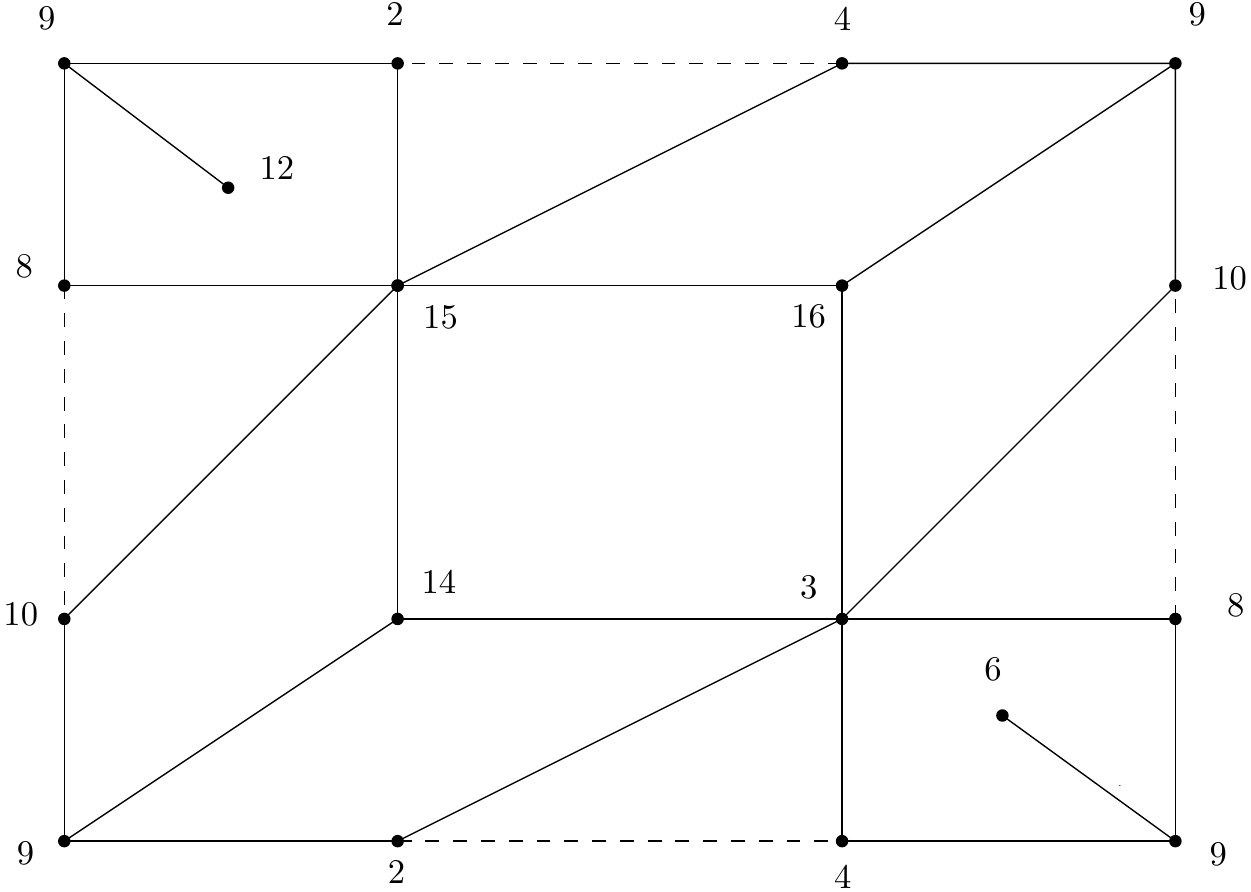}
    \caption{Embedding of ${\mathcal{D}(\mathbb{Z}_{18})}$ in $\mathbb{N}_2$.}
    \label{Z18c}
\end{figure}

\vspace{.05cm}
     \; \; \; \textbf{Subcase-4.1.2(b):}  \emph{$P_2$ contains two cyclic subgroups $H$ and $K$ of order $9$ such that $|H\cap K|=3$}. Let $P_1 =\langle x \rangle$, $H=\langle y \rangle$ and $K=\langle z\rangle$. Suppose $x_1, x_2\in H\cap K$ such that $o(x_1)=o(x_2)=3$. Then
     $$P_1H= \{e, x, x_1, x_2, y, y^2, y^4, y^5, y^7, y^8,  xx_1, xx_2, xy, xy^2, xy^4, xy^5, xy^7, xy^8\}, \ \text{and}$$  $$P_1K= \{e, x, x_1, x_2, z, z^2, z^4, z^5, z^7, z^8, xx_1, xx_2, xz, xz^2, xz^4, xz^5, xz^7, xz^8\}$$ are maximal cyclic subgroups of order $18$ (see {\rm \cite[Lemma 2.11]{a.chattopadhyay2021minimal}}). Consider the sets $S=\{x,xx_1,xx_2\}$ and $T=\{y,y^2,y^4,y^5,y^7,y^8,z,z^2,z^4,z^5,z^7,z^8\}$. Let $x'\in S$ and $y'\in T$. Notice that $x'\sim y'$ in $\a$. Also, neither $o(x')\vert o(y')$ nor  $o(y')\vert o(x')$. It follows that $x'\nsim y'$ in $\mathcal{P}(G)$ and so $x'\sim y'$ in $\d$. Thus, the subgraph induced by  $S\cup T$ has a subgraph isomorphic to $K_{3,12}$. It implies that $\gamma(\d)\geq 3$ and $\overline{\gamma}(\d)\geq 5$.

\vspace{.05cm}
     \textbf{Subcase-4.1.3:}  \emph{$|P_1|=2$ and $\mathrm{exp}(P_2)=3^{\alpha}$, where $\alpha \geq 3$}. Then there exists an element $y\in P_2$ such that $o(y)=27$. Let $P_1=\langle x\rangle $. Observe that $\langle xy\rangle$ is a cyclic subgroup of order $54$ in $G$. Consider the sets $S=\{s\in \langle xy\rangle : o(s)=18\}$ and $T=\{t\in \langle xy\rangle : o(t)=27\}$. Suppose $x'\in S$ and $y'\in T$. Observe that $x'\sim y'$ in $\a$ and $x'\nsim y'$ in $\mathcal{P}(G)$ (cf. Remark \ref{order not divide}). It follows that $x'\sim y'$ in $\d$.  Thus, $\d$ has a subgraph isomorphic to $K_{|S|,|T|}$. Since $|S|=6$ and $|T|=18$, we obtain $\gamma(\d)\geq 16$ and $\overline{\gamma}(\d)\geq 32$. 

     \vspace{.05cm}
      \textbf{Subcase-4.1.4:} \emph{$|P_1|=2$ and $\mathrm{exp}(P_2)=p_2^{\alpha}$, where $p_2\geq 5$, $\alpha \geq 2$}. Then there exists an element $y\in P_2$ such that $o(y)=p_2^2$. Let $P_1=\langle x\rangle$. Observe that $\langle xy\rangle$ is a cyclic subgroup of order $2p_2^2$ in $G$. Consider the sets $S=\{s\in \langle xy\rangle : o(s)=2p_2\}$ and $T=\{t\in \langle xy\rangle : o(t)=p_2^2\}$. Similar to Subcase 4.1.3, we obtain a subgraph of $\d$ which is isomorphic to $K_{|S|,|T|}$. Since $p_2\geq 5$, we have $|S|\geq 4$ and $|T|\geq 20$. It follows that $\gamma(\d)\geq 9$ and $\overline{\gamma}(\d)\geq 18$. 

     \vspace{.05cm}
      \textbf{Subcase-4.2.1:} \emph{$|P_1|=4$ and $|P_2|=3$}. Then either $G\cong \mathbb{Z}_4\times  \mathbb{Z}_3$ or $G\cong \mathbb{Z}_2\times \mathbb{Z}_2 \times \mathbb{Z}_3$. By Theorem \ref{PlanarD}, in both of these cases, $\d$ is a planar graph. Consequently, $\gamma(\d)=0=\overline{\gamma}(\d)$.

\vspace{.05cm}
     \textbf{Subcase-4.2.2:} \emph{$|P_1|=4$ and $|P_2|=5$}. Then either $G\cong \mathbb{Z}_{4}\times \mathbb{Z}_{5} \cong \mathbb{Z}_{20}$ or $G\cong \mathbb{Z}_2\times \mathbb{Z}_2 \times  \mathbb{Z}_5$. If $G\cong \mathbb{Z}_{20}$, then by Theorem \ref{PlanarD}, $\gamma(\d)\geq 1$ and $\overline{\gamma}(\d)\geq 1$. A genus $1$ and cross-cap $1$ drawings of $\mathcal{D}(\mathbb{Z}_{20})$ are given in Figures \ref{Z20}
 and \ref{Z20c}, respectively.  
  \begin{figure}[ht]
    \centering
    \includegraphics[width=0.50\textwidth]{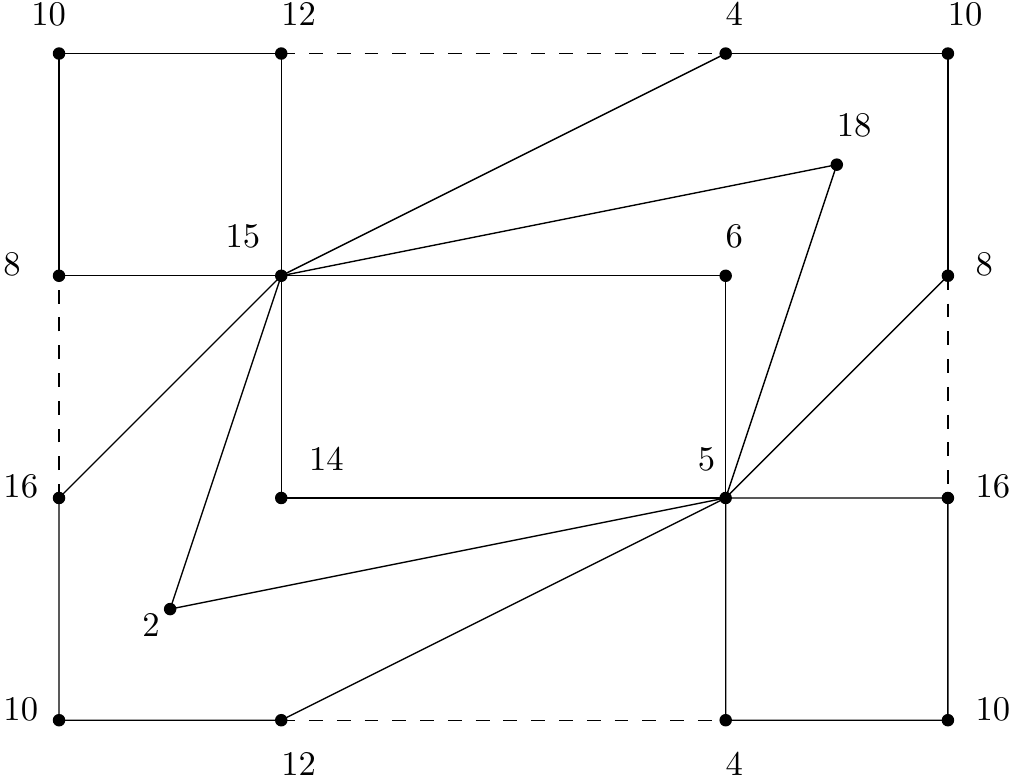}
    \caption{Embedding of ${\mathcal{D}(\mathbb{Z}_{20})}$ in $\mathbb{S}_1$.}
    \label{Z20}
\end{figure}
 \begin{figure}[ht]
    \centering
    \includegraphics[width=0.50\textwidth]{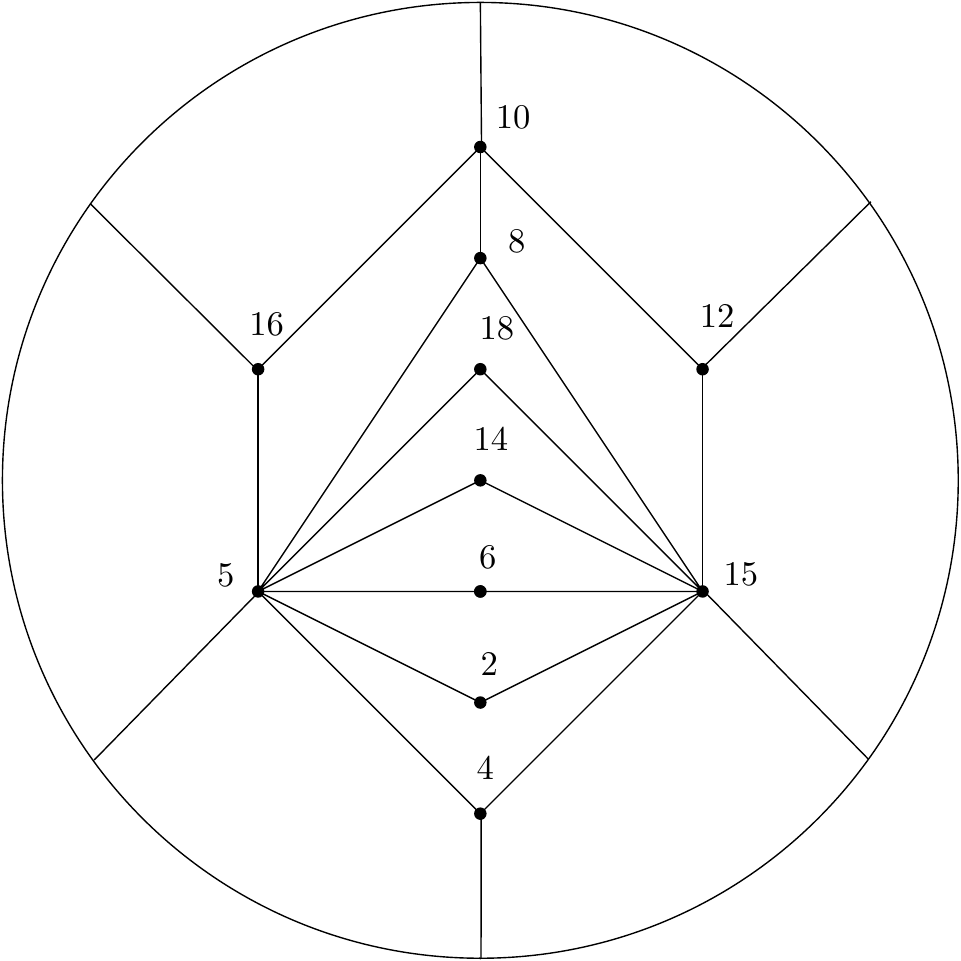}
    \caption{Embedding of ${\mathcal{D}(\mathbb{Z}_{20})}$ in $\mathbb{N}_1$.}
    \label{Z20c}
\end{figure}

If $G\cong \mathbb{Z}_2\times \mathbb{Z}_2 \times  \mathbb{Z}_5$, then by Lemma \ref{lemma 3}, $\d \cong K_{3,4}$. Consequently, $\gamma(\d)=1$ and $\overline{\gamma}(\d)=1$.

\vspace{.05cm}
      \textbf{Subcase-4.2.3:} \emph{$|P_1|=4$ and $|P_2|=7$}. Then $G$ is isomorphic to  $ \mathbb{Z}_{4}\times \mathbb{Z}_{7}$ or $ \mathbb{Z}_2\times \mathbb{Z}_2 \times  \mathbb{Z}_7$. If $G \cong \mathbb{Z}_{4}\times \mathbb{Z}_{7}$, then $\d$ contains a subgraph which is isomorphic to $K_{3,6}$ (see Lemma \ref{lemma 2}). Consequently, $\gamma(\d)\geq 1$ and $\overline{\gamma}(\d)\geq 2$. A genus $1$ and cross-cap $2$ drawings of $\mathcal{D}(\mathbb{Z}_{28})$ are given in Figure \ref{Z28}
 and \ref{Z28c}, respectively.  
 
\begin{figure}[ht]
    \centering
    \includegraphics[width=0.50\textwidth]{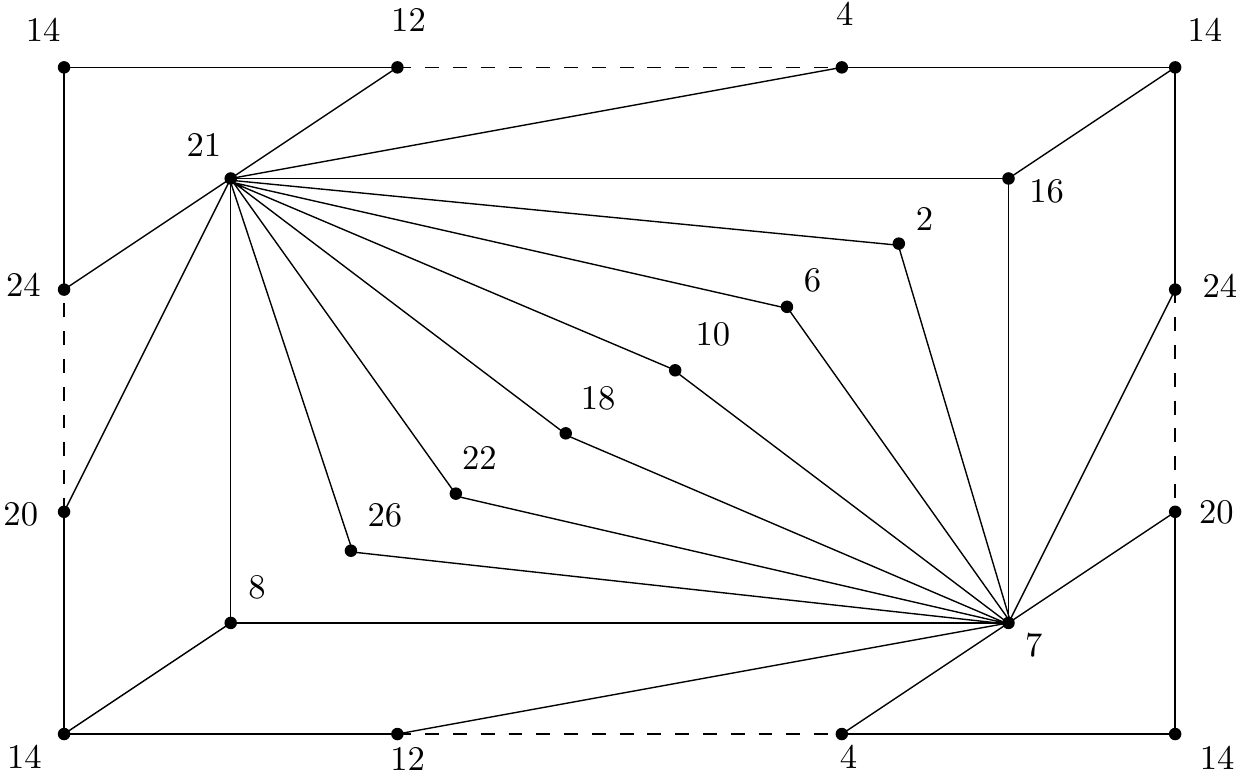}
    \caption{Embedding of ${\mathcal{D}(\mathbb{Z}_{28})}$ in $\mathbb{S}_1$.}
    \label{Z28}
\end{figure}
 \begin{figure}[ht]
    \centering
    \includegraphics[width=0.50\textwidth]{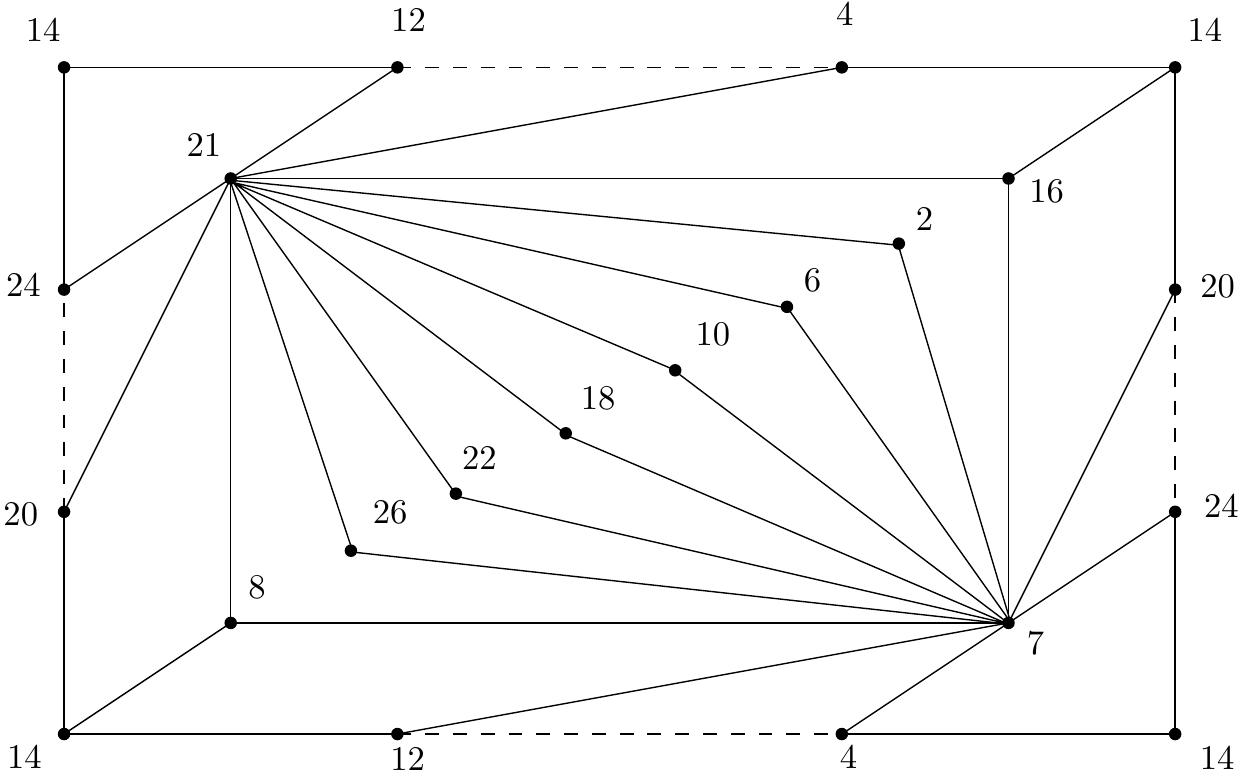}
    \caption{Embedding of ${\mathcal{D}(\mathbb{Z}_{28})}$ in $\mathbb{N}_2$.}
    \label{Z28c}
\end{figure}

If $G\cong \mathbb{Z}_2\times \mathbb{Z}_2 \times  \mathbb{Z}_7$, then by Lemma \ref{lemma 2}, $\d \cong K_{3,6}$. It follows that $\gamma(\d)= 1$ and $\overline{\gamma}(\d)= 2$.

\vspace{.05cm}
       \textbf{Subcase-4.2.4:} \emph{$|P_1|=4$ and $|P_2|=9$}. Then either $P_1\cong \mathbb{Z}_2\times \mathbb{Z}_2$ or $P_1\cong \mathbb{Z}_4$. Also, $P_2$ is isomorphic to $\mathbb{Z}_3\times \mathbb{Z}_3$ or $ \mathbb{Z}_9$. Consequently, $G$ is isomorphic to one of the groups: $\mathbb{Z}_{36}$, $\mathbb{Z}_2\times \mathbb{Z}_2\times \mathbb{Z}_9$, $\mathbb{Z}_4\times \mathbb{Z}_3\times \mathbb{Z}_3$, $\mathbb{Z}_2 \times \mathbb{Z}_2 \times \mathbb{Z}_3\times \mathbb{Z}_3$. If $G\cong \mathbb{Z}_{36}$, then $\a$ is a complete graph. Consider the set $S= \{x\in G : o(x)=9\}$ and $T=\{y\in G : o(y)\in \{2,4,6\}\}$. Let $x'\in S$ and $y'\in T$. Clearly, $x'\sim y'$ in $\a$ and by Remark \ref{order not divide}, $x'\nsim y'$ in $\mathcal{P}(G)$. Thus, $x'\sim y'$ in $\d$ and so  $\d$ contains a subgraph which is isomorphic to $K_{|S|,|T|}$. Since $|S|=6$ and $|T|=5$, we obtain $\gamma(\d)\geq 3$ and $\overline{\gamma}(\d)\geq 6$.\\
       If $G\cong \mathbb{Z}_2\times \mathbb{Z}_2\times \mathbb{Z}_9$, then note that $G$ has exactly $3$ maximal cyclic subgroups $H_1=\langle (1,0,1)\rangle$, $H_2=\langle (0,1,1) \rangle$ and $H_3= \langle (1,1,1)\rangle$. Consider the set $S=\{x\in G: o(x)=9\}$ and $T=\{y\in G : o(y)\in \{2,6\}\}$. Then observe that $S\subseteq (H_1\cap H_2\cap H_3)$, $|S|=6$ and $|T|=9$.  Suppose $x'\in S$ and $y'\in T$ are arbitrary elements. Then $x',y'\in H_i$ for some $i\in \{1,2,3\}$. It follows that $x'\sim y'$ in $\a$. Consequently, $x'\sim y'$ in $\d$ (cf. Remark \ref{order not divide}). Thus,  $\d$ contains a subgraph isomorphic to $K_{6,9}$. Hence, $\gamma(\d)\geq 7$ and $\overline{\gamma}(\d)\geq 14$.\\
       If $G\cong \mathbb{Z}_4\times \mathbb{Z}_3\times \mathbb{Z}_3$, then by Lemma \ref{lemma 1}, $\d$ contains a subgraph which is isomorphic to $K_{3,8}$. Thus, $\gamma(\d)\geq 2$ and $\overline{\gamma}(\d)\geq 3$. A genus $2$ drawing of $\d$ is given in Figure \ref{Z4Z3Z3}. \\
       \begin{figure}[ht]
    \centering
    \includegraphics[width=0.50\textwidth]{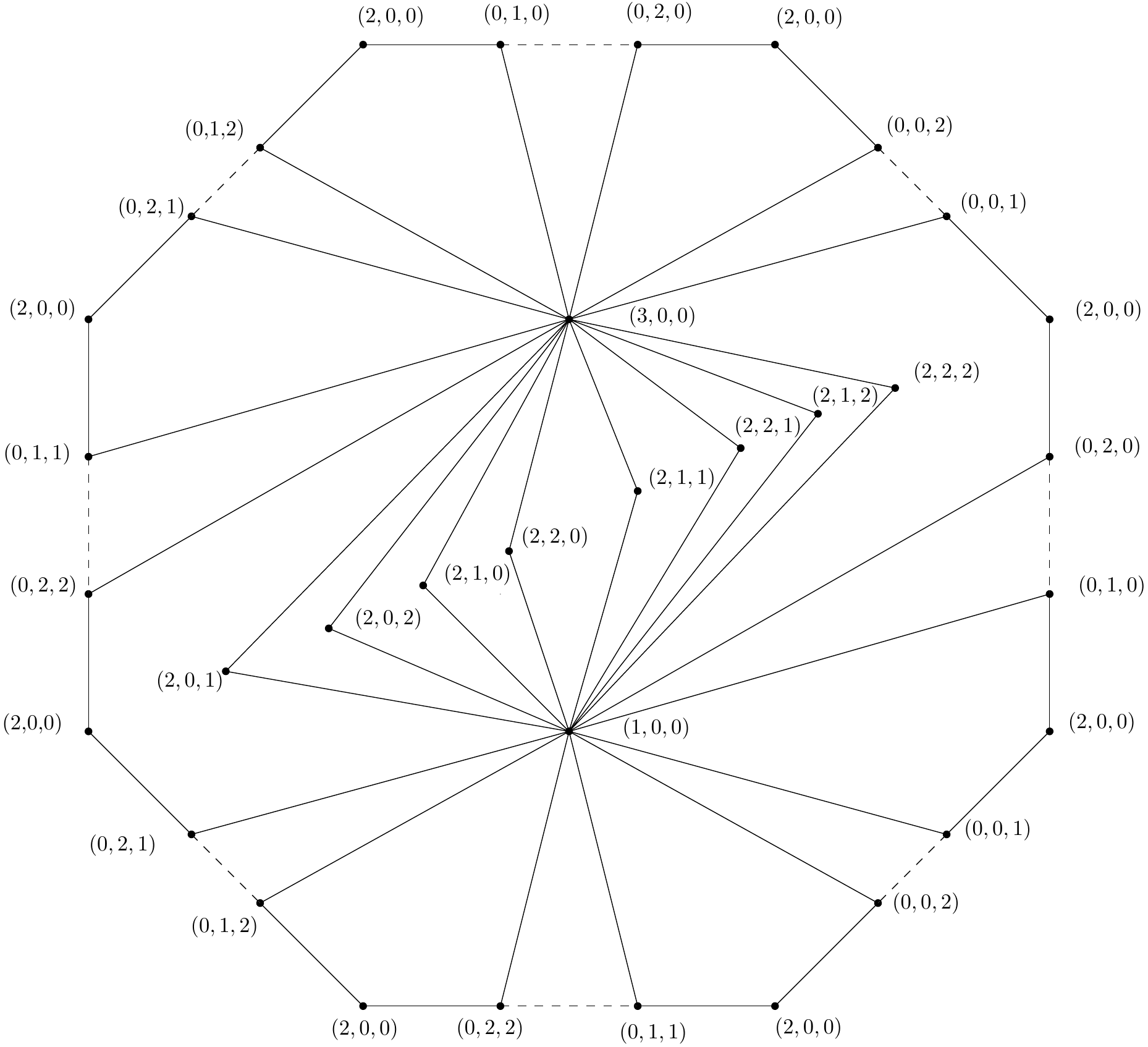}
    \caption{Embedding of ${\mathcal{D}(\mathbb{Z}_4\times \mathbb{Z}_3\times \mathbb{Z}_3)}$ in $\mathbb{S}_2$.}
    \label{Z4Z3Z3}
\end{figure} 
     If  $G \cong \mathbb{Z}_2 \times \mathbb{Z}_2 \times \mathbb{Z}_3\times \mathbb{Z}_3$ then by Lemma \ref{lemma 2}, we have $\d \cong K_{3,8}$. Consequently, $\gamma(\d)= 2$ and $\overline{\gamma}(\d)= 3$.
       
\vspace{.05cm}
       \textbf{Subcase-4.2.5:} \emph{$|P_1|=4$ and $|P_2|=11$}. Thus, either $G\cong \mathbb{Z}_{4}\times \mathbb{Z}_{11} \cong \mathbb{Z}_{44}$ or $G\cong \mathbb{Z}_2\times \mathbb{Z}_2 \times  \mathbb{Z}_{11}$. If $G \cong \mathbb{Z}_{4}\times \mathbb{Z}_{11}$, then the subgraph of $\d$ induced by the set $(P_1\cup P_2)\setminus \{e\}$ contains a subgraph isomorphic to $K_{3,10}$. Consequently, $\gamma(\d)\geq 2$ and $\overline{\gamma}(\d)\geq 4$. A genus $2$ drawing of ${\mathcal{D}(\mathbb{Z}_{44})}$ is given in Figure \ref{Z44}.  \\ 
 If $G\cong \mathbb{Z}_2\times \mathbb{Z}_2 \times  \mathbb{Z}_{11}$, then $\d \cong K_{3,10}$ (cf. Lemma \ref{lemma 2}) and so $\gamma(\d)= 2$, $\overline{\gamma}(\d)= 4$.
 
\begin{figure}[ht]
    \centering
    \includegraphics[width=0.50\textwidth]{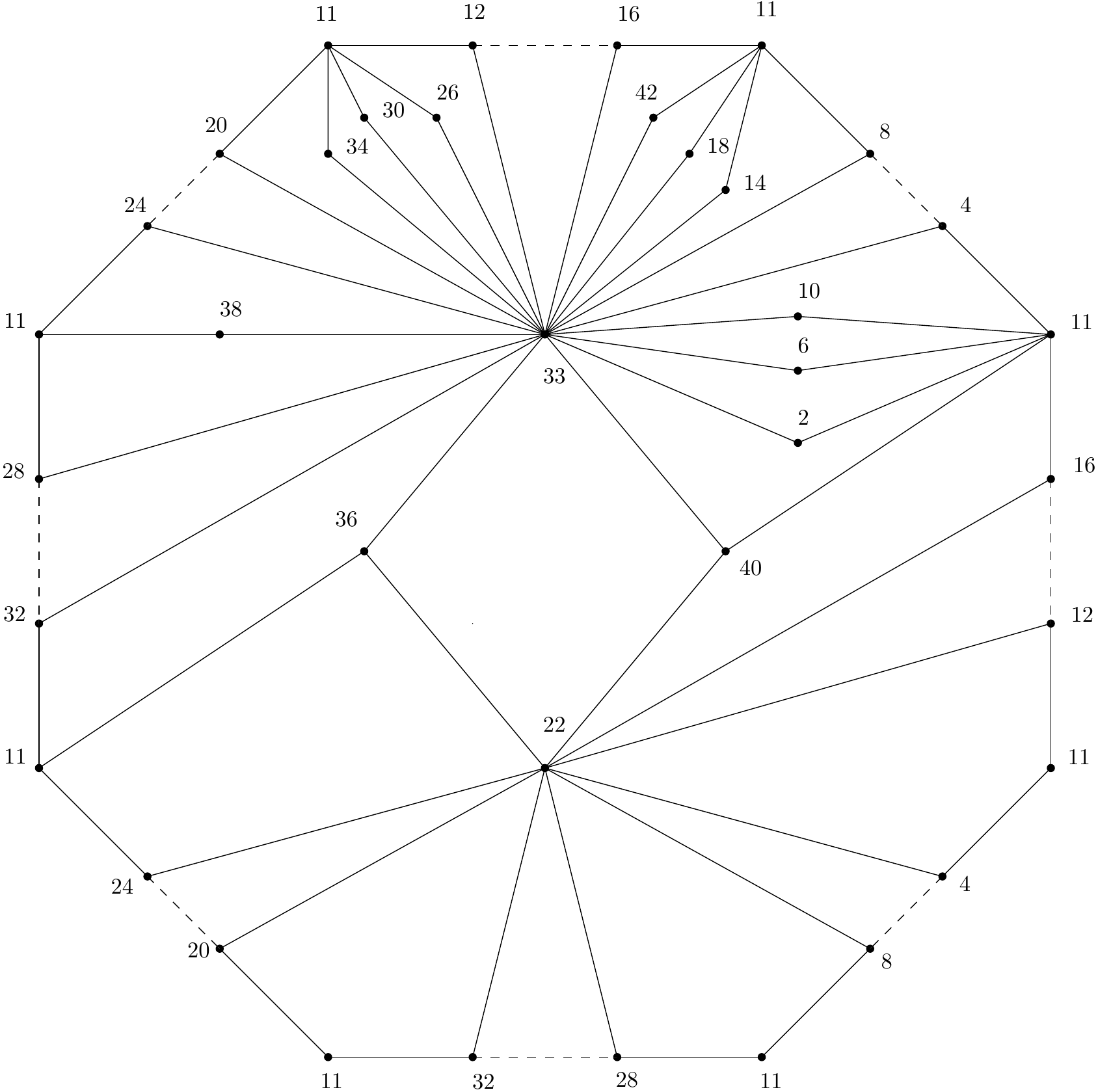}
    \caption{Embedding of ${\mathcal{D}(\mathbb{Z}_{44})}$ in $\mathbb{S}_2$.}
    \label{Z44}
\end{figure}

\vspace{.05cm}
     \textbf{Subcase-4.2.6:} \emph{$|P_1|=4$ and $|P_2|=p_2$, where $p_2\geq 13$}. By Lemma \ref{lemma 1}, $\d$ contains a subgraph isomorphic to $K_{3,12}$. Thus, $\gamma(\d)\geq 3$ and $\overline{\gamma}(\d)\geq 5$.

\vspace{.05cm}
     \textbf{Subcase-4.2.7:} \emph{$|P_1|=4$ and $|P_2|=p_2^{\alpha}$, where $p_2\geq 5$, $\alpha \geq 2$ and $\mathrm{exp}(P_2)=p_2$}. By Lemma \ref{lemma 1}, $\d$ contains a subgraph isomorphic to $K_{3,24}$. Thus, $\gamma(\d)\geq 6$ and $\overline{\gamma}(\d)\geq 11$.

    \vspace{.05cm}
       \textbf{Subcase-4.2.8:} \emph{$|P_1|=4$ and $|P_2|=3^{\alpha}$, where $\alpha \geq 3$}. It implies that the minimum number of non-identity elements in $P_1$ and $P_2$ are $3$ and $26$, respectively. Consequently, by Lemma \ref{lemma 1}, $\d$ contains  a subgraph isomorphic to $K_{3,26}$. It follows that $\gamma(\d)\geq 6$ and $\overline{\gamma}(\d)\geq 12$.

 \vspace{.05cm}
       \textbf{Subcase-4.3.1:} \emph{$|P_1|= 2^{\alpha}$ and $|P_2|=3^{\beta}$, where $\alpha \geq 3$, $\beta \geq 2$}. By Lemma \ref{lemma 1}, the subgraph induced by the set $(P_1\cup P_2)\setminus \{e\}$ contains a subgraph isomorphic to $K_{7,8}$. It follows that $\gamma(\d)\geq 8$ and $\overline{\gamma}(\d)\geq 15$.

\vspace{.05cm}
       \textbf{Subcase-4.3.2:} \emph{$|P_1|= 2^{\alpha}$ $(\alpha \geq 3)$ with $\mathrm{exp}(P_1)=2$ and $|P_2|=3$}. By Theorem \ref{PlanarD}, the graph $\d$ is planar and so $\gamma(\d)=0=\overline{\gamma}(\d)$.

\vspace{.05cm}
        \textbf{Subcase-4.3.3:} \emph{$|P_1|= 2^{\alpha}$  $(\alpha \geq 3)$ with $\mathrm{exp}(P_1)=4$ and $|P_2|=3$}. Consider $P_2=\langle x \rangle$. Further, suppose that $P_1$ has $t \ (\geq 1)$ maximal cyclic subgroups of order $4$ and $s \ (\geq 0)$ maximal cyclic subgroups of order $2$. Consider the maximal cyclic subgroups $M_i=\langle y_i\rangle$, where $1\leq i \leq t$, of order $4$. If $s\geq 1$, then consider $M_j'=\langle z_j\rangle$, where $1\leq j \leq s$, as the maximal cyclic subgroup of order $2$. Consequently, maximal cyclic subgroups of order $12$ in $G$ are of the form $M_iP_2=\{e,x,x^2,y_i,y_i^2,y_i^3,y_ix,y_i^2x,y_i^3x,y_ix^2,y_i^2x^2,y_i^3x^2\}$. Also, maximal cyclic subgroups of order $6$ are of the form $M_j'P_2=\{e,x,x^2,z_j, z_jx,z_jx^2\}$ (see {\rm \cite[Lemma 2.11]{a.chattopadhyay2021minimal}}). Notice that $P_2$ contained in every maximal cyclic subgroup of $G$. Also, the identity element $e$ and the generators of $M_iP_2$ and $M_j'P_2$ do not belong to the vertex set of $\d$ (cf. Proposition \ref{proposition 2.1}). 
        
        If $t=1$, then by Lemma \ref{lemma 4}, we have $P_1=\mathbb{D}_8$. Thus, $G\cong \mathbb{D}_8\times \mathbb{Z}_3$. By Theorem \ref{PlanarD}, $\d$ becomes planar  and so $\gamma(\d)= 0= \overline{\gamma}(\d)$. We may now suppose that $t>1$. In view of Lemma \ref{lemma 3}, now we discuss this subcase into the following six further subcases. 
        \vspace{.05cm}
        
\; \; \;        \textbf{Subcase-4.3.3(a):} \emph{$P_1$ contains two maximal cyclic subgroups $H$ and $K$ of order $4$ such that $|H\cap K|=2$, and the intersection of any other pair of maximal cyclic subgroups of $P_1$ is trivial}. Without loss of generality, assume that $H=M_1$ and $K=M_2$.
        Since $|M_1\cap M_2|=2$, we get $M_1P_2\cap M_2P_2=\{e,x,x^2,y_1^2,y_1^2x, y_1^2x^2\}$. Consider the set $S=\{y\in M_1P_2\cup M_2P_2 : o(y)=4\}$ and $T=\{z\in M_1P_1\cup M_2P_2 : o(z)\in \{3,6\}\}$. Let $x'\in S$ and $y'\in T$. Then $x',y'\in M_iP_2$ for some $i\in \{1,2\}$. Thus, $x'\sim y'$ in $\a$. Consequently, $x'\sim y'$ in $\d$ (cf. Remark \ref{order not divide}). Since $|S|=4$ and $|T|=4$, the subgraph induced by $S\cup T$ is isomorphic to $K_{4,4}$. Thus,  $\gamma(\d)\geq 1$ and $\overline{\gamma}(\d)\geq 2$. A genus $1$ and cross-cap $2$ drawings of $\d$ are given in Figures \ref{PZ_3g} and \ref{PZ_3c}, respectively. The graph $G_3$, given in Figure \ref{F_i}, can be inserted in the face $F$ of Figure \ref{PZ_3g} and \ref{PZ_3c}, respectively, without edge crossings.
          \begin{figure}[ht]
    \centering
    \includegraphics[width=0.50\textwidth]{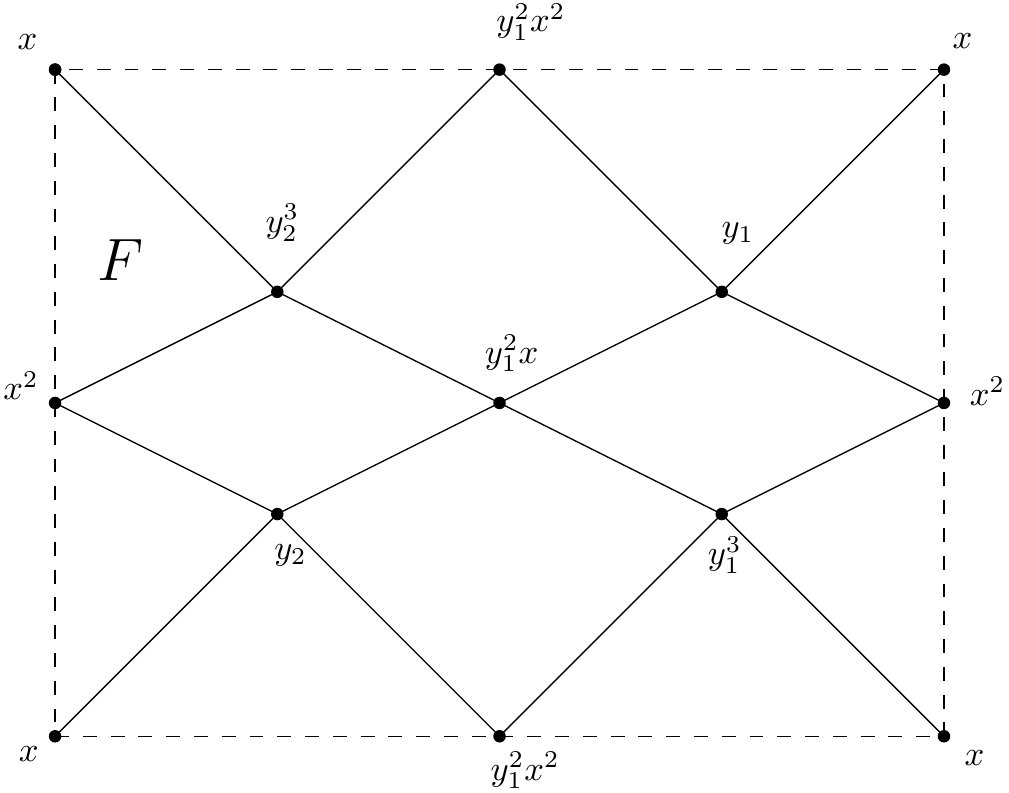}
    \caption{Embedding of ${\mathcal{D}(P_1\times P_2)}$, where $P_1$ and $P_2$ are according to Subcase-4.3.3(a),  in $\mathbb{S}_1$.}
    \label{PZ_3g}
\end{figure} 

          \begin{figure}[ht]
    \centering
    \includegraphics[width=0.50\textwidth]{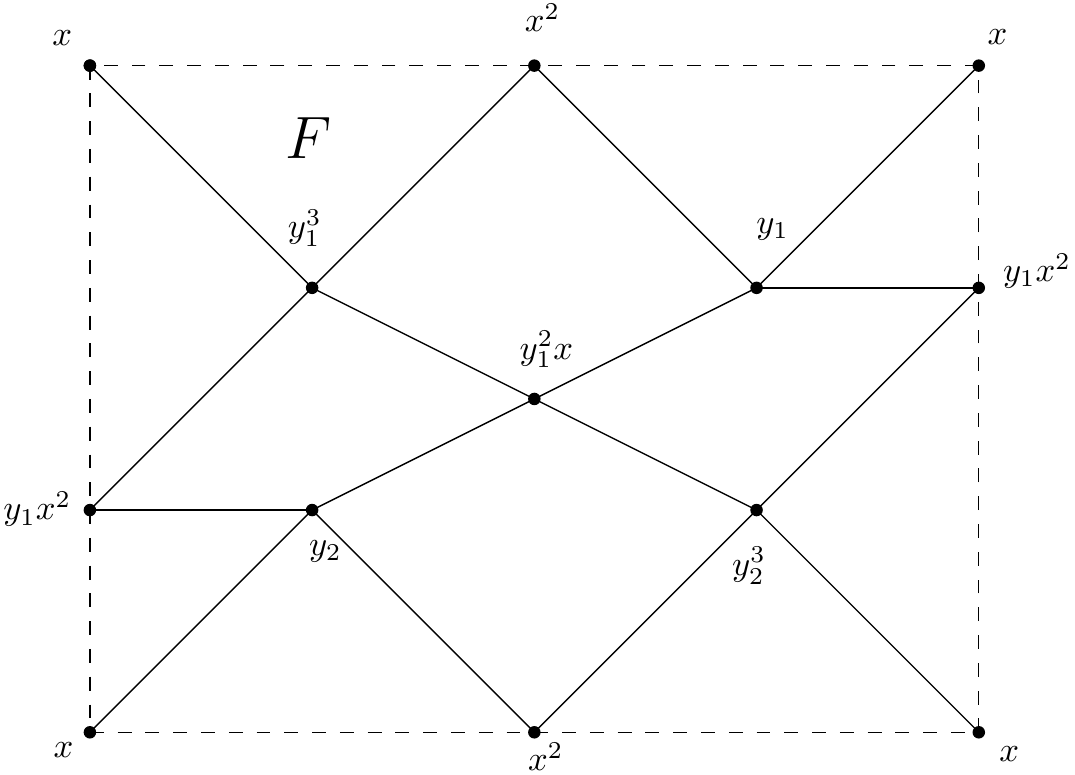}
    \caption{Embedding of ${\mathcal{D}(P_1\times P_2)}$, where $P_1$ and $P_2$ are according to Subcase-4.3.3(a),  in $\mathbb{N}_2$.}
    \label{PZ_3c}
\end{figure} 

\vspace{.05cm}
  \; \; \; \textbf{Subcase-4.3.3(b):}  \emph{$P_1$ contains four maximal cyclic subgroups $H_1, H_2, H_3$ and $H_4$ of order $4$ such that $|H_1\cap H_2|=|H_3\cap H_4|=2$, and the intersection of any other pair of maximal cyclic subgroups of $P_1$ is trivial}. Without loss of generality, assume that $H_i=M_i$ for $1\leq i \leq 4$.
  Now similar to the Subcase-4.3.3(a), we get a subgraph $\Gamma'$ of $\d$, which is isomorphic to $K_{4,4}$. Moreover, $|M_3\cap M_4|=2$. It implies that $M_3P_2\cap M_4P_2=\{e,x,x^2,y_3^2,y_3^2x, y_3^2x^2\}$. 
  Now to embed $\d$ through $\Gamma'$ in $\mathbb{S}_1$, first we insert the vertices $y_3,y_3^3,y_4,y_4^3, y_3^2x, y_3^2x^2$ and their incident edges. Since $y_3\sim y_3^2x\sim y_3^3\sim y_3^2x^2\sim y_4$ and $y_3^2x\sim y_4^3\sim  y_3^2x^2$; all these vertices must be inserted in the same face $F'$. Note that the vertices $y_3,y_3^3,y_4$ and $y_4^3$ are adjacent to both the vertices $x$ and $x^2$ (see Proposition \ref{proposition 2.1}). Consequently, the face $F'$ must contain the vertices $x$ and $x^2$. After inserting the vertices $y_3,y_3^3,y_4,y_4^3$ and their incident edges, it is impossible to insert the vertices $y_3^2x, y_3^2x^2$ without edge crossing (see Figure \ref{The face $F'$}). Thus, $\gamma(\d)\geq 2$. A genus $2$ drawing of $\d$ is given in Figure \ref{pz3g2}, and the subgraph $G_5$, (given in Figure \ref{F_i}) can be inserted in the face $F$. By the similar argument used earlier in this subcase,  any embedding of $\d$ in $\mathbb{N}_2$ is also not possible without edge crossings. Hence, $\overline{\gamma}(\d)\geq 3$.

         \begin{figure}[ht]
    \centering
    \includegraphics[width=0.75\textwidth]{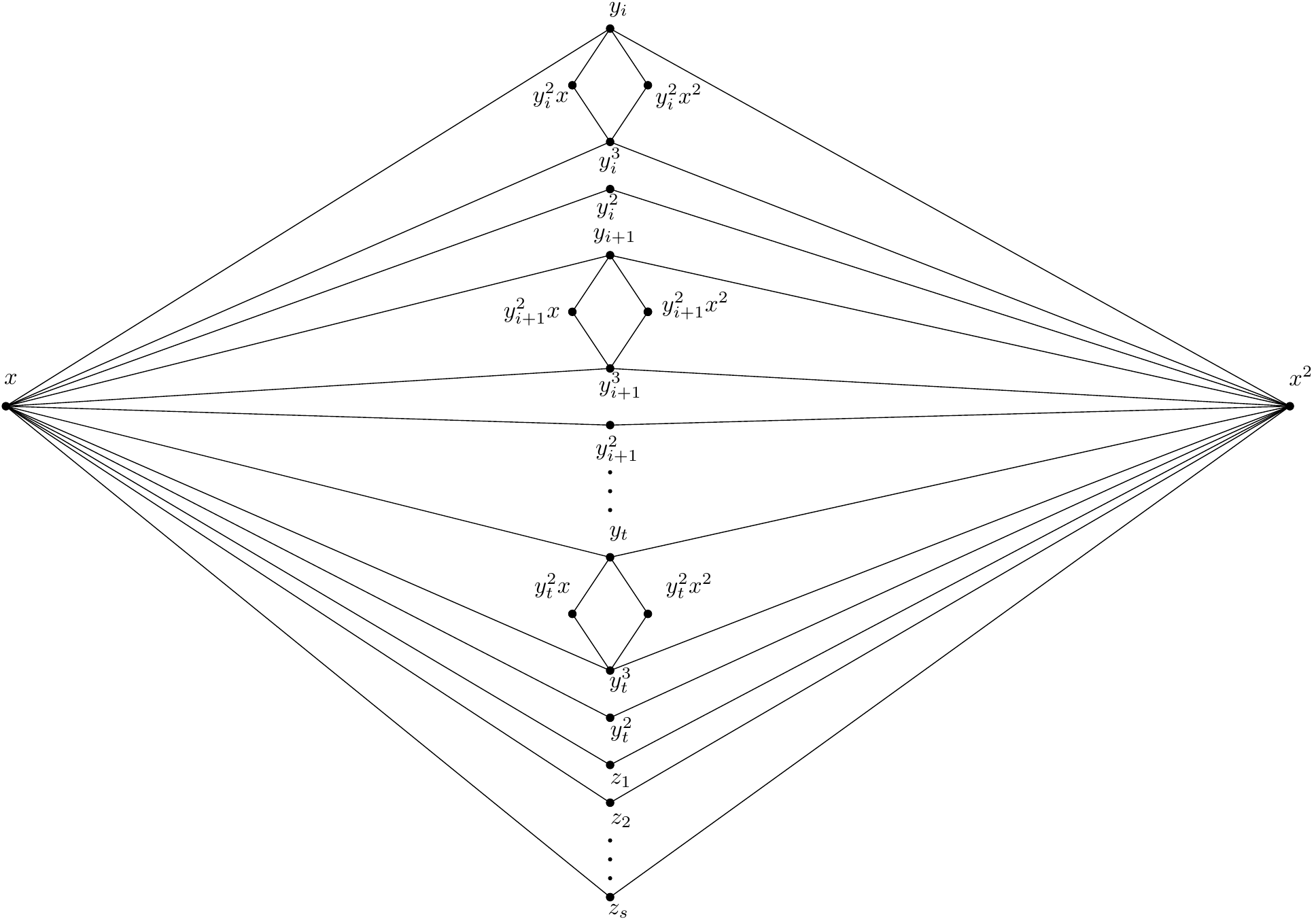}
    \caption{The subgraph $G_i$ of $\d$ induced by the set $(V(\d)\setminus (\bigcup \limits_{j=1}^{i-1} M_jP_2))\cup \{x,x^2\}$.}
    \label{F_i}
\end{figure} 

    \begin{figure}[ht]
    \centering
    \includegraphics[width=0.5\textwidth]{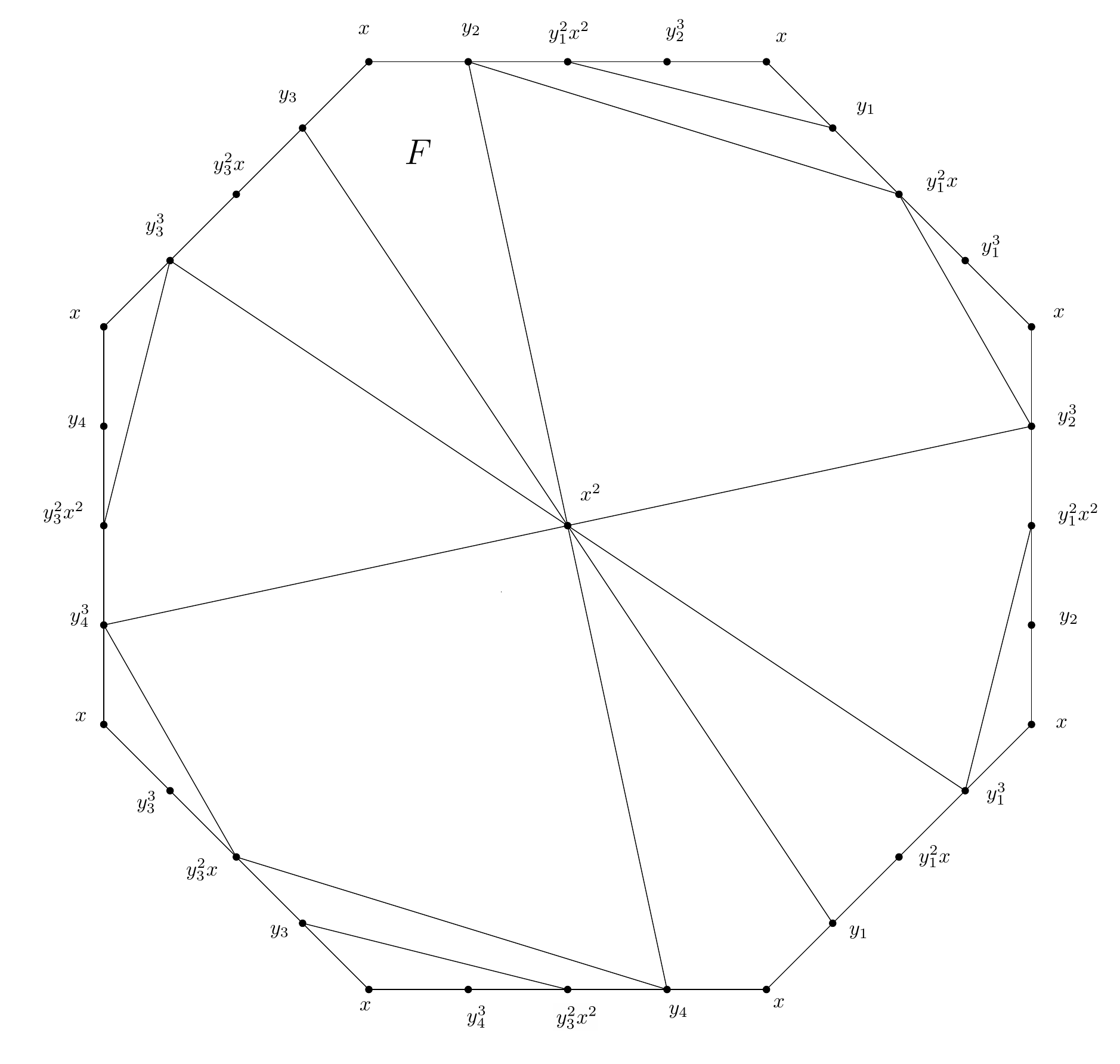}
    \caption{Embedding of ${\mathcal{D}(P_1\times P_2)}$, where $P_1$ and $P_2$ are according to Subcase-4.3.3(b),  in $\mathbb{S}_2$.}
    \label{pz3g2}
\end{figure}

         \begin{figure}[ht]
    \centering
    \includegraphics[width=0.5\textwidth]{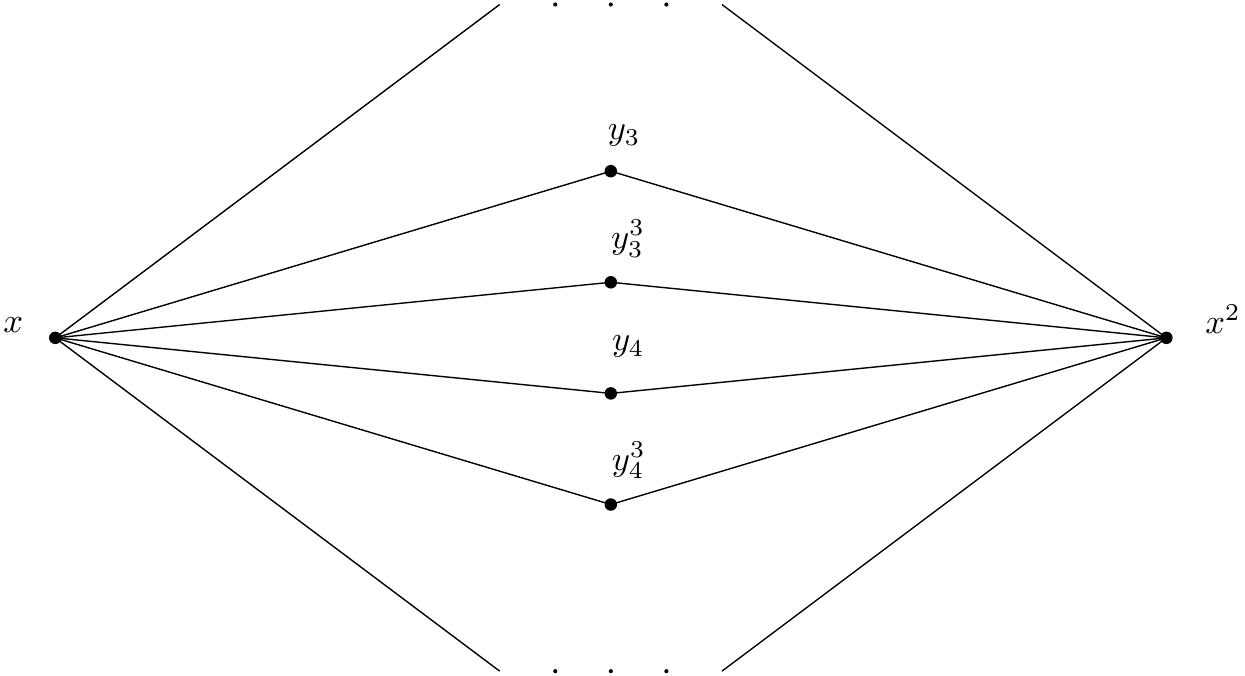}
    \caption{The face $F'$}
    \label{The face $F'$}
\end{figure}

\vspace{.05cm}
          \; \; \; \textbf{Subcase-4.3.3(c):}  \emph{$P_1$ contains six maximal cyclic subgroups $H_1, H_2, H_3, H_4, H_5$ and $H_6$ of order $4$ such that $|H_1\cap H_2|=|H_3\cap H_4|=|H_5\cap H_6|=2$, and the intersection of any three maximal cyclic subgroups of $P_1$ is trivial}. Without loss of generality, assume that $H_i=M_i$ for each $i \in [6]$.
          Similar to the Subcase-4.3.3(b), we obtain a subgraph $\Gamma''$ of $\d$, induced by the set  $V(\d)\cap (M_1P_2\cup M_2P_2\cup M_3P_2\cup M_4P_2)$, which can not be embedded in $\mathbb{S}_1$ and $\mathbb{N}_2$. It follows that $\gamma(\d)\geq 2$ and $\overline{\gamma}(\d)\geq 3$. Also, we have $M_5P_2\cap M_6P_2=\{e,x,x^2,y_5^2,y_5^2x, y_5^2x^2\}$. Suppose $\gamma(\d)= 2$. To embed $\d$ in $\mathbb{S}_2$, first we insert the vertices $y_5,y_5^3,y_6,y_6^3, y_5^2x, y_5^2x^2$ and their incident edges in genus $2$ drawing of $\Gamma'$. By the similar argument used in Subcase-4.3.3(b) (by taking $y_5$ in place of $y_3$, and $y_4$ in place of $y_6$) it is impossible to insert these vertices in $\mathbb{S}_2$ without edge crossings.
          Thus, $\gamma(\d)\geq 3$.

\vspace{.05cm}
\; \; \;  \textbf{Subcase-4.3.3(d):} \emph{$P_1$ contains three maximal cyclic subgroups $H_1, H_2$ and $H_3$ of order $4$ such that $|H_1\cap H_2\cap H_3|=2$, and the intersection of any other pair of maximal cyclic subgroups of $P_1$ is trivial}.  Without loss of generality, assume that $H_i=M_i$ for $1\leq i \leq 3$.
  Since $|M_1\cap M_2\cap M_3|=2$, we obtain $M_1P_2\cap M_2P_2\cap M_3P_2=\{e,x,x^2,y_1^2,y_1^2x, y_1^2x^2\}$. Consider the set $S=\{y\in M_1P_2\cup M_2P_2 \cup M_3P_2 : o(y)=4\}$ and $T=\{z\in M_1P_2\cup M_2P_2\cup M_3P_2: o(z)\in \{3,6\}\}$. Note that  the subgraph $\Gamma'$, induced by the set $S\cup T$ is isomorphic to $K_{6,4}$. Thus, $\overline{\gamma}(\d)\geq 4$ and ${\gamma}(\d)\geq 2$. 
Now we show that $\d$ can be embedded in $\mathbb{S}_2$ without edge crossing. Since $\Gamma'$ is a bipartite graph, it implies that each face of $\Gamma'$ is of even length at least $4$ in $\mathbb{S}_2$. Consequently, each face must contain at least two vertices of each partite set of $\Gamma'$. It follows that there exists an embedding of $\Gamma'$ in $\mathbb{S}_2$ such that the face $F_2$ contains the vertices $x$ and $x^2$. Now one can embed $\d$ in $\mathbb{S}_2$ through $\Gamma'$ by inserting the subgraph $G_4$ into $F_2$. Therefore,  ${\gamma}(\d)= 2$.

\vspace{.05cm}
 \; \; \; \textbf{Subcase-4.3.3(e):} \emph{$P_1$ contains five maximal cyclic subgroups $H_1, H_2,H_3, H_4$ and $H_5$ of order $4$ such that $|H_1\cap H_2\cap H_3|=2=|H_4\cap H_5|$, and the intersection of any four maximal cyclic subgroups of $P_1$ is trivial}. Without loss of generality, assume that $H_i=M_i$ for each $i\in [5]$.
 Similar to the Subcase-4.3.3(d), we obtain a subgraph $\Gamma'$ of $\d$, induced by the set  $V(\d)\cap (M_1P_2\cup M_2P_2\cup M_3P_2)$, which can not be embedded in $\mathbb{S}_1$ and $\mathbb{N}_3$ without edge crossing. It follows that $\gamma(\d)\geq 2$ and $\overline{\gamma}(\d)\geq 3$. Moreover, $|M_4\cap M_5|=2$ implies that $M_4P_2\cap M_5P_2=\{e,x,x^2,y_4^2,y_4^2x, y_4^2x^2\}$. Suppose $\gamma(\d)=2$.  Now to embed $\d$ in $\mathbb{S}_2$, first we insert the vertices $y_4,y_4^3,y_5,y_5^3, y_4^2x, y_4^2x^2$ and their incident edges in genus $2$ drawing of $\Gamma'$.  By the similar argument used in the Subcase-4.3.3(b) (by taking $y_4$ in place of $y_3$, and $y_5$ in place of $y_6$) it is impossible to insert these vertices in $\mathbb{S}_2$ without edge crossings.
 It follows that $\gamma(\d)\geq 3$.

\vspace{.05cm}
 \; \; \; \textbf{Subcase-4.3.3(f):}
\emph{$P_1$ contains four maximal cyclic subgroups $H_1, H_2, H_3$ and $H_4$ of order $4$ such that $|H_1\cap H_2\cap H_3\cap H_4|=2$}. Without loss of generality, assume that $H_i=M_i$ for each $ i\in [4]$.
Since $|M_1\cap M_2\cap M_3\cap M_4|=2$, we obtain $M_1P_2\cap M_2P_2\cap M_3P_2\cap M_4P_2=\{e,x,x^2,y_1^2,y_1^2x, y_1^2x^2\}$. Consider the set $S=\{y\in M_1P_2\cup M_2P_2 \cup M_3P_2\cup M_4P_2 : o(y)=4\}$ and $T=\{z\in M_1P_2\cup M_2P_2\cup M_3P_2\cup M_4P_2: o(z)\in \{3,6\}\}$. Note that  the subgraph of $\d$ induced by $S\cup T$ has a subgraph isomorphic to $K_{8,4}$. Thus, $\overline{\gamma}(\d)\geq 3$ and ${\gamma}(\d)\geq 6$. 
  
\vspace{.05cm}
       \textbf{Subcase-4.3.4:} \emph{$|P_1|= 2^{\alpha}$ and $|P_2|=3$, where $\alpha \geq 3$ and $\mathrm{exp}(P_1)\geq 8$}. Then there exists an element $y\in P_2$ such that $o(y)=8$. Suppose $x\in P_1$ such that $o(x)=3$. Notice that $\langle xy\rangle$ is a cyclic subgroup of order $24$ in $G$. Consider the sets $S=\{s\in \langle xy\rangle : o(s)=8\}$ and $T=\{t\in \langle xy\rangle : o(t)\in \{3,6,12\}\}$. Let $x' \in S$ and $y' \in T$. Observe that, $x'\sim y'$ in $\a$. Consequently, by Remark \ref{order not divide}, $x'\sim y'$ in $\d$. Thus, $\d$ contains a subgraph isomorphic to $K_{|S|,|T|}$. Since $|S|= 4$ and $|T|= 8$, we obtain $\gamma(\d)\geq 3$ and $\overline{\gamma}(\d)\geq 6$.

\vspace{.05cm}
       \textbf{Subcase-4.3.5:} \emph{$|P_1|= 2^{\alpha}$ and $|P_2|=p_2$ such that $\alpha \geq 3$, $p_2\geq 5$}. Then $P_1$ has at least $7$ non-identity elements and $P_2$ has at least $4$ non-identity elements. Consequently, by Lemma \ref{lemma 1}, $\d$ contains a subgraph which is isomorphic to $K_{4,7}$. It follows that $\gamma(\d)\geq 3$ and $\overline{\gamma}(\d)\geq 5$.

This completes our proof.

\section*{Declarations}

\textbf{Funding}: The first author gratefully acknowledge for providing financial support to CSIR  (09/719(0110)/2019-EMR-I) government of India. The second author wishes to  acknowledge the support of Core Research Grant (CRG/2022/001142) funded by  SERB. 

\vspace{.3cm}
\textbf{Conflicts of interest/Competing interests}: There is no conflict of interest regarding the publishing of this paper. 

\vspace{.3cm}
\textbf{Availability of data and material (data transparency)}: Not applicable.

\vspace{.3cm}
\textbf{Code availability (software application or custom code)}: Not applicable.

\vspace{1cm}
\noindent
{\bf Parveen\textsuperscript{\normalfont 1}, \bf Jitender Kumar\textsuperscript{\normalfont 1}}
\bigskip

\noindent{\bf Addresses}:

\vspace{5pt}

\end{document}